\documentclass[11pt]{amsart}
\usepackage{amsmath,amssymb,amsthm}
\usepackage[latin1]{inputenc}
\usepackage{tabularx,multicol,array}
\usepackage{graphicx,float,psfrag}
 \usepackage{stmaryrd,mathrsfs}
\usepackage{url}
\usepackage{tikz}
\usepackage{caption}
\usepackage{subcaption}
\usepackage{mathrsfs}
\usepackage{accents}
\usepackage[numbers]{natbib}

\newcommand{\reff}[1]{(\ref{#1})}

\theoremstyle{plain}
\newtheorem{theo}{Theorem}[section]
\newtheorem*{theo*}{Theorem}
\newtheorem{cor}[theo]{Corollary}
\newtheorem{prop}[theo]{Proposition}
\newtheorem{lem}[theo]{Lemma}
\newtheorem{defi}[theo]{Definition}
\theoremstyle{remark}
\newtheorem{rem}[theo]{Remark}

\newtheorem*{ex*}{Example}

\newcommand{\ca}{{\mathcal A}}

\newcommand{\cc}{{\mathcal C}}
\newcommand{\cd}{{\mathcal D}}

\newcommand{\cf}{{\mathcal F}}

\newcommand{\cj}{{\mathbb{J}}}
\newcommand{\cjj}{{\mathcal J}}

\newcommand{\cl}{{\mathcal L}}

\newcommand{\cs}{{\mathcal S}}

\newcommand{\N}{{\mathbb N}}
\renewcommand{\P}{{\mathbb P}}

\newcommand{\R}{{\mathbb R}}

\newcommand{\ind}{{\bf 1}}

\newcommand{\val}[1]{\mathop{\left| #1 \right|}\nolimits}
\newcommand{\inv}[1]{\mathop{\frac{1}{ #1}}\nolimits}
\newcommand{\expp}[1]{\mathop {\mathrm{e}^{ #1}}}

\renewcommand{\phi}{\varphi}
\renewcommand{\epsilon}{\varepsilon}

\title{Maximum entropy distribution of order statistics with given marginals}
\date{\today}

 \author{Cristina Butucea}
 \address{
 Cristina Butucea,
 Université Paris-Est, LAMA (UPE-MLV), 77455 Marne La Vallée, France.}
 \email{cristina.butucea@univ-mlv.fr}

 \author{Jean-François Delmas}
 \address{
 Jean-Fran\c cois Delmas,
 Université Paris-Est, CERMICS (ENPC), 77455 Marne La Vallée, France.}
 \email{delmas@cermics.enpc.fr}

 \author{Anne Dutfoy}
 \address{
 Anne Dutfoy,
 EDF Research \& Development, Industrial Risk Management Department, 92141 Clamart Cedex, France.}
 \email{anne.dutfoy@edf.fr}

 \author{Richard Fischer}
 \address{
 Richard Fischer,
 Université Paris-Est, CERMICS (ENPC), 77455 Marne La Vallée, France\\
 EDF Research \& Development, Industrial Risk Management Department, 92141 Clamart Cedex, France.}
 \email{fischerr@cermics.enpc.fr}

\begin{document}

\thanks{This work is partially supported by the French ``Agence Nationale de
 la Recherche'',CIFRE n° 1531/2012, and by EDF Research \& Development, Industrial Risk Management Department}



 \begin{abstract}
   We consider distributions of ordered random vectors with given one-dimensional
   marginal distributions. We give an elementary necessary and sufficient
   condition for the existence of such a distribution with finite entropy.
   In this case, we give explicitly the density of the unique distribution
   which achieves the maximal entropy and compute the value of its entropy.
   This density is the unique one which has a product form on its support
   and the given one-dimensional marginals. The proof relies on the
   study of copulas with given one-dimensional marginal distributions for its
   order statistics.
\end{abstract}

\maketitle

 \section{Introduction} \label{sec:intro}

 Order statistics,  an almost  surely non-decreasing sequence  of random
 variables, have  received a lot  of attention  due to the  diversity of
 possible  applications. If  $X=(X_1,\hdots,X_d)$  is a  $d$-dimensional
 random  vector,  then  its order  statistics  $X^{OS}=(X_{(1)},  \ldots,
 X_{(d)})$ corresponds  to the permutation  of the components  of $X$  in the
 non-decreasing order,  so that $X_{(1)}  \leq X_{(2)} \leq  \hdots \leq
 X_{(d)}$.  The  components of  the  underlying  random vector  $X$  are
 usually, but  not necessarily, independent and  identically distributed
 (i.i.d.). Special attention has been  given to extreme values $X_{(1)}$
 and  $X_{(d)}$,  the  range  $X_{(d)}-X_{(1)}$, or  the  median  value.
 Direct  application of  the distribution  of the  $k$-th largest  order
 statistic  occurs  in  various  fields, such  as  climatology,  extreme
 events, reliability, insurance, financial  mathematics. We refer to the
 monographs  of \citeauthor*{david1970order}~\cite{david1970order}  and \citeauthor*{arnold1992first}~\cite{arnold1992first}  for a
 general overview on the subject  of order statistics. We are interested
 in the  dependence structure  of order  statistics, which  has received
 great attention  when the underlying  random vector is i.i.d.   and for
 the   non   i.i.d.     case   as   well.    In    the   i.i.d.    case,
 \citeauthor*{bickel1967some}~\cite{bickel1967some}  showed   that  any  two  order   statistics  are
 positively  correlated.   The  copula  of  the  joint  distribution  of
 $X_{(1)}$ and $X_{(d)}$ is  derived in \citeauthor*{schmitz2004revealing}~\cite{schmitz2004revealing} with
 exact  formulas  for  Kendall's   $\tau$  and  Spearman's  $\rho$.   In
 \citeauthor*{averous2005dependence}~\cite{averous2005dependence}, it  is shown  that the dependence  of the
 $j$-th order statistic  on the $i$-th order statistic  decreases as the
 distance  between $i$  and  $j$ increases  according  to the  bivariate
 monotone  regression dependence  ordering.  The  copula connecting  the
 limit  distribution  of  the   two  largest  order  statistics,  called
 bi-extremal  copula,  is  given   by  \citeauthor*{de2007limiting}~\cite{de2007limiting}  with  some
 additional  properties.  Exact  expressions  for Pearson's  correlation
 coefficient, Kendall's $\tau$  and Spearman's $\rho$ for  any two order
 statistics  are  obtained  in  \citeauthor*{navarro2010study}~\cite{navarro2010study}.   For  the  non
 i.i.d. case,  \citeauthor*{kim1990dependence}~\cite{kim1990dependence} shows  that some pairs  of order
 statistics  can  be negatively  correlated,  if  the underlying  random
 vector  is  sufficiently  negatively  dependent.   Positive  dependence
 measures    for    two    order   statistics    are    considered    in
 \citeauthor*{boland1996bivariate}~\cite{boland1996bivariate}  when the  underlying  random variables  are
 independent but  arbitrarily distributed  or when they  are identically
 distributed but not independent. A  generalization of these results for
 multivariate dependence properties is given by \citeauthor*{hu2008dependence}~\cite{hu2008dependence}.
 See also \citeauthor*{dubhashi2008note}~\cite{dubhashi2008note} for  conditional distribution of order
 statistics.\\

 Here, we focus  on the cumulative distribution function  (cdf) of order
 statistics without  referring to an underlying  distribution.  That is,
 we consider random vectors $X=(X_1,\hdots,X_d) \in \R^d$ such that a.s.
 $X_1\leq  \cdots \leq  X_d$  and we  suppose  that the  one-dimensional
 marginal distributions $\mathbf{F}=(\mathbf{F}_i, 1 \leq i \leq d)$ are
 given,  where $\mathbf{F}_i$  is the  cdf  of $X_i$.   A necessary  and
 sufficient condition for the existence of a joint distribution of order
 statistics with one-dimensional marginals $\mathbf{F}$ is that they are
 stochastically ordered, that is:
\begin{equation} \label{eq:stoch_order}
     \mathbf{F}_{i-1}(x) \geq \mathbf{F}_i(x) \quad \text{ for all } 2
     \leq i \leq d, x \in \R.
\end{equation}
      With  the  marginals  fixed,  the joint  distribution  of  the  order
   statistics  can be  characterized  by the  connecting  copula of  the
   random  vector,  which contains  all  information  on the  dependence
   structure  of  the order  statistics.   Copulas  of order  statistics
   derived   from  an   underlying  i.i.d. sample   were  considered   in
   \cite{averous2005dependence} in order to calculate measures of concordance
   between any two pairs of order statistics.
   For order statistics derived from a general parent distribution,
   \citeauthor*{NavarroSpizzichino2010}~\cite{NavarroSpizzichino2010}
   shows that the copula of the order statistics depends on the
   marginals  and the copula of the parent distribution through an
   exchangeable copula and the average of the marginals.  Construction of
   some copula  of order statistics  with given marginals were  given in
   \citeauthor*{lebrun2014copulas}~\cite{lebrun2014copulas}.\\

   Our aim is to find the  cdf of order  statistics of dimension $d$ with
   fixed marginals which maximizes  the differential entropy $H$ defined
   as, for a cdf $F$ with density $f$:
 \[
    H(F)=-\int f \log f,
 \]
 and $H(F)=-\infty$ if $F$  does not have a density. If  $Z$ is a random
 variable with cdf $F$, we shall  use the convention $H(Z)=H(F)$.  In an
 information-theoretic interpretation, the  maximum entropy distribution
 is the least  informative among order statistics  with given marginals.
 This problem  appears in models  where the one-dimensional  marginals are
 well  known (either  from  different experimentation  or from  physical
 models) but the dependence structure is unknown. In \citeauthor*{butucea2013maximum}~\cite{butucea2013maximum},
 we gave, when it exists, the maximum entropy distribution of $(X_1, \hdots, X_d)$
 such that $X_i$ is uniformly distributed on $[0,1]$ for $1\leq i \leq d$
 and the distribution of $X_{(d)}=\max_{1\leq i \leq d} X_i$ is given.

 For a $d$-dimensional  random variable $X=(X_1, \hdots,  X_d)$ with cdf
 $F$ and copula $C_F$, the entropy of $F$ can be decomposed into the sum
 of the  entropy of  its one-dimensional marginals  plus the  entropy of
 $C_F$ (see \citeauthor*{Zhao2011628}~\cite{Zhao2011628}):
 \[
    H(F)=\sum_{i=1}^d H(\mathbf{F}_i) + H(C_F),
 \]
 where  $\mathbf{F}_i$ is  the cdf  of $X_i$.   In our  case, since  the
 marginals  $\mathbf{F}=(\mathbf{F}_i,  1 \leq  i  \leq  d)$ are  fixed,
 maximizing  the  entropy  of  the   joint  distribution  $F$  of  an order
 statistics  is  equivalent to  maximizing  the  entropy of  its  copula
 $C_F$. Therefore  we shall  find the  maximum entropy  copula for
 order statistics with fixed marginal distributions.
To solve this question, we introduce the functional of $\mathbf{F}$
satisfying \reff{eq:stoch_order}:
\[
 \cj(\mathbf{F})=\sum_{i=2}^d \int_{\R}
 \mathbf{F}_{i}(dt)\,
 \val{\log\left(\mathbf{F}_{i-1}(t)-\mathbf{F}_{i}(t) \right)} .
\]

The main result of this paper is given by Theorem \ref{theo:f*} which
we reproduce here.
\begin{theo*}
Let $\mathbf{F}=(\mathbf{F}_i, 1 \leq i \leq d)$ be a $d$-dimensional vector of
 cdf's on $\R$ satisfying \reff{eq:stoch_order}.
\begin{itemize}
\item    If   there    exists    $1\leq   i    \leq    d$   such    that
  $H(\mathbf{F}_i)=-\infty$, or if $\cj({\mathbf{F}})  = + \infty$, then
  for all  cdf $F$ of   order  statistic with  one-dimensional marginals
  $\mathbf{F}$, we have $H(F)=-\infty $.
\item If  $H(\mathbf{F}_i) > -  \infty$ for all $1  \leq i \leq  d$, and
  $\cj({\mathbf{F}})  <  +  \infty$,  then there  exists  a  unique  cdf
  $F_\mathbf{F}$  of an order   statistic  with one-dimensional  marginals
  $\mathbf{F}$    such     that    $H(F_\mathbf{F})>-\infty     $    and
  $H(F_\mathbf{F})\geq H(F) $ for any cdf  $F$ of  order  statistic with
  one-dimensional marginals $\mathbf{F}$. Furthermore we have:
\[
H(F_{\mathbf{F}})=d-1  + \sum_{i=1}^d H({\mathbf{F}_i}) -
\cj(\mathbf{F}),
\]
and $F_\mathbf{F}$ has density $f_\mathbf{F}$ given
for $x=(x_1, \hdots, x_d) \in \R^d$ by:
   \[
      f_{\mathbf{F}}(x)=
            \mathbf{f}_{1} (x_1) \prod_{i=2}^d \frac{\mathbf{f}_{i}(x_i)}{{\mathbf{F}}_{i-1}\left(x_i \right)-{\mathbf{F}}_{i}(x_i)} \exp\left(-\int_{x_{i-1}}^{x_i}
           \frac{\mathbf{f}_{i}(s)}{{\mathbf{F}}_{i-1}(s)-{\mathbf{F}}_{i}(s) } \, ds \right) \ind_{L^\mathbf{F}}(x),
   \]
   where $\mathbf{f}_i$ is the density  function of $\mathbf{F}_i$ and $
   L^\mathbf{F}\subset  \R^d$  is  the  set of  ordered  vectors  $(x_1,
   \ldots,  x_d)$,  that is  $x_1\leq  \cdots  \leq  x_d$, such  that  $
   \mathbf{F}_{i-1}(t)> \mathbf{F}_i(t)$  for all $t\in  (x_{i-1}, x_i)$
   and $2\leq  i\leq d$.  (Notice that this  last condition  is automatically
   fulfilled if $\mathbf{F}_{i-1}> \mathbf{F}_i$.)
\end{itemize}
\end{theo*}
The   function   $f_{\mathbf{F}}$   may   be  well   defined   even   if
$\cj({\mathbf{F}}) =  + \infty$ and  it might  even be a density, see
Remark \ref{rem:f_F=density}.  However  in this case the  entropy of the
corresponding cdf is  infinite. The  density $f_{\mathbf{F}}$ has  a product
form on  the domain  $L^\mathbf{F}$: $f_{\mathbf{F}}(x)  = \prod_{i=1}^d
p_i(x_i) \ind_{L^\mathbf{F}}(x)$ for $x=(x_1, \hdots, x_d) \in \R^d$ and
some measurable functions $p_i$. Corollary \ref{cor:f*_uni_prod} asserts
that $f_{\mathbf{F}}$  is the only density  which has a product  form on
the domain $L^\mathbf{F}$ and  whose one-dimensional marginals are given
by $\mathbf{F}$.  This characterization  will be  used in  a forthcoming
paper on nonparametric estimation of $f_{\mathbf{F}}$.

\begin{ex*}
  We  consider the  following example.  Let $+\infty  >\lambda_1> \cdots
  >\lambda_d>0$ and for $1\leq i\leq d$ let $\mathbf{F}_i$ be the cdf of
  the  exponential  distribution  with mean  $1/\lambda_i$  and  density
  $\mathbf{f}_i(t)=\lambda_i \expp{-\lambda_i t}\ind_{\{t>0\}}$.  Notice
  that $\mathbf{F}_{i-1}> \mathbf{F}_i$ on $(0,+\infty )$, so that $L^\mathbf{F}=\{ (x_1,\hdots,x_d) \in \R^d; 0 \leq x_1 \leq \hdots \leq x_d\}$. It is easy to
  check that $\cj(\mathbf{F})<+\infty  $ with $\mathbf{F}=(\mathbf{F}_i,
  1\leq  i\leq  d)$.  Elementary  computations  yield  that the  maximum
  entropy density  of the  order statistic  $(X_1, \ldots,  X_d)$, where
  $X_i$ has distribution $\mathbf{F}_i$, is given by:
\[
f_\mathbf{F}(x_1, \ldots, x_d)
=\ind_{L^\mathbf{F}}(x) \, \lambda_1 \expp{-\Delta_{2}
  x_1} \left(1-\expp{-\Delta_{2}
      x_1}\right)^{\lambda_{2}/\Delta_{2}}\prod_{i=2} ^d \lambda_i \expp{-\Delta_{i+1}
  x_i}\frac{\left(1-\expp{-\Delta_{i+1}
      x_i}\right)^{\lambda_{i+1}/\Delta_{i+1}} }
{\left(1-\expp{-\Delta_{i}
      x_i}\right)^{\lambda_{i-1}/\Delta_{i}} },
\]
where $\Delta_i=\lambda_{i-1}-\lambda_i$ for $1\leq i\leq d+1$ and $\lambda_{d+1}=0$.

In the particular case $\lambda_i=(d-i+1) \lambda$ for some $\lambda>0$,
we get:
\[
f_\mathbf{F}(x_1, \ldots, x_d)
=\ind_{L^\mathbf{F}}(x)\,  d! \, \lambda ^ d \expp{-\lambda x_1}
(1-\expp{-\lambda x_1})^{d-1} \prod_{i=2} ^d
\frac{\expp{-\lambda
      x_i}}
{\left(1-\expp{-\lambda
      x_i}\right)^2}\cdot
\]

By considering the change of variable $u_i=1-\expp{-\lambda x_i}$, we
get the following result. For $1\leq i\leq d$ let $\mathbf{F}_i$ be the cdf of
the $\beta(1, d-i+1)$ distribution   with   density
  $\mathbf{f}_i(t)=(d-i+1) (1-t)^{d-i}\ind_{(0,1)}(t)$.  Notice
  that $\mathbf{F}_{i-1}> \mathbf{F}_i$ on $(0,1 )$. The  maximum
  entropy density  of the  order statistic  $(U_1, \ldots,  U_d)$, where
  $U_i$ has distribution $\mathbf{F}_i$, is given by:
\[
f_\mathbf{F}(u_1, \ldots, u_d)
=\ind_{\{0<u_1<\cdots <u_d<1\}}\,  d! \,
u_1^{d-1} \, \prod_{i=2} ^d
\inv{u_i^2}\cdot
\]
Elementary computations give $H(F_\mathbf{F})=-\log(d!)+ 2d
-(d+1)\sum_{i=1}^d (1/i)$.
\end{ex*}

In  order  to  prove  Theorem  \ref{theo:f*},  we  first  introduce  the
multidiagonal of  a copula.   For a  copula $C$ on  $\R^d$ and  a random
vector $U=(U_{1},\hdots,  U_{d})$ with  cdf $C$,  we consider  its order
statistics $U^{OS}=(U_{(1)}, \hdots,  U_{(d)})$ and define $\delta_{(i)}
$ the  one-dimensional cdf of $U_{(i)}$  for $1\leq i\leq d$.   Then the
multidiagonal of $C$ is defined as  $ \delta_C = (\delta_{(i)}, 1 \leq i
\leq d)$. This provides a generalization  of the diagonal section of the
copula   $C$   which   corresponds   to  the   cdf   $\delta_{(d)}$   of
$U_{(d)}=\max_{1\leq  i\leq d}  U_i$.  The  maximum entropy  copula with
fixed  diagonal  section  is given  in  \cite{butucea2013maximum}.   The
necessary and  sufficient condition for  a $d$-tuple $\delta$ to  be the
multidiagonal  of  a  (absolutely  continuous)  copula  is  provided  by
\citeauthor*{jaworski2008distributions}~\cite{jaworski2008distributions}.    In    order   to    prove   Theorem
\ref{theo:f*}, we  first establish a  one-to-one map between the  set of
copulas  of  order  statistics  with given  marginals  $\mathbf{F}$  and
symmetric copulas with a  fixed multidiagonal $\delta^\mathbf{F}$, which
only  depends  on  $\mathbf{F}$,   see  Lemma  \ref{lem:one_to_one}  and
Proposition \ref{prop:char_cfsym}.  Furthermore,  this map preserves the
absolute  continuity of  the copula,  as well  as the  entropy up  to an
additive  constant,  which depends  only  on  the fixed  one-dimensional
marginals $\mathbf{F}$,  see Proposition  \ref{prop:H(C)_H(SFC)}.  Then,
under a  necessary and  sufficient condition on  the multidiagonal
$\delta$,  we find  the  maximum  entropy  copula with
a given general multidiagonal $\delta$ and provide an explicit formula for
its density, see
Theorem  \ref{theo:spec}.   (Notice  Theorem \ref{theo:spec},  which  is
first established to prove Theorem \ref{theo:f*}, can in fact be seen as
a special case of Theorem \ref{theo:f*}.) The proof relies on the theory
of optimization under infinite dimensional constraints developed by
\citeauthor*{borwein1994entropy}~\cite{borwein1994entropy}. \\

The   rest  of   the  paper   is  organized   as  follows.   In  Section
\ref{sec:not-def},  we  introduce  the  basic  notations  and  give  the
definition   of    the   objects   used   in    later   parts.   Section
\ref{sec:connection} describes  the connection between copulas  of order
statistics  with  fixed  marginals,  and symmetric  copulas  with  fixed
multidiagonals.  In Section  \ref{sec:max_ent_copula}, we  determine the
maximum  entropy copula  with  fixed multidiagonal.  Since  we obtain  a
symmetric copula as a result, this is also the maximum entropy symmetric
copula with fixed  multidiagonal. In     Section
\ref{sec:max_entr_dens_ord_stat}, we use the one-to-one map
between the two sets of copulas established
in Section \ref{sec:connection}  to  give
the maximum  entropy copula  of order     statistics     with     fixed     marginals
. We finally obtain  the density of the maximum entropy  distribution for
order     statistics     with     fixed     marginals  by composing the
maximum entropy copula with the marginals.  Section \ref{sec:proof}  contains the
detailed proofs of  Theorem   \ref{theo:spec} and other results from
Section \ref{sec:max_ent_copula}. Section \ref{sec:not-appendix} collects
the main notations of the paper to facilitate reading.

 \section{Notations and definitions} \label{sec:not-def}

\subsection{Notations in $\R^d$ and generalized inverse}
\label{sec:not-J-1}
For a  Borel set $A\subset  \R^d$, we  write $\val{A}$ for  its Lebesgue
measure.   For  $x=(x_1,  \ldots,  x_d)\in \R^d$  and  $y=(y_1,  \ldots,
y_d)\in  \R ^d$,  we  write $x\leq  y$  if $x_i\leq  y_i$
 for  all $1\leq  i\leq  d$.  We  define $\min  x=\min
\{x_i, \, 1\leq  i\leq d\}$ and $\max x=\max \{x_i,  \, 1\leq i\leq d\}$
for $x=(x_1,  \ldots, x_d)\in  \R^d$. If $J$  is a  real-valued function
defined  on $\R$,  we  set $J(x)=(J(x_1),  \ldots,  J(x_d))$.  We  shall
consider the following subsets of $ \R^d$:
\[
S=\{(x_1, \ldots, x_d)\in
   \R^d, x_1\leq \dots \leq  x_d\} \quad\text{and}\quad
\triangle=S\cap I^d,
\]
with $I=[0,1]$. In what follows, usually $x,y$ will belongs to $\R^d$,  and $s,t$ to $\R$
or $I$. For a set $A \subset \R$, we note by $A^c=\R \setminus A$ its complementary set.

If $J$ is a bounded non-decreasing c\`{a}d-l\`{a}g function defined on $\R$.
Its generalized inverse $J^{-1} $ is given
by $J^{-1}(t)= \inf\{s \in \R; J(s)\geq t\}$, for $t\in
\R$, with the convention that $\inf \emptyset=+\infty $ and $ \inf \R = - \infty$. We have for $s,t\in  \R$:
\begin{equation}
   \label{eq:inv}
J(t)\geq s \Leftrightarrow t\geq J^{-1} (s),
\quad J^{-1} \circ J(t)\leq t \quad\text{and}\quad
J\circ J^{-1}\circ J(t)=J(t) .
\end{equation}
We define  the set of
points where $J$ is increasing on their left:
\begin{equation}
   \label{eq:def-Ig}
I_g(J)=\{t\in \R; u<t \Leftrightarrow J(u)<J(t)\}.
\end{equation}
We have:
\begin{equation} \label{eq:IgJc}
  \ind_{(I_g(J))^c} \, dJ =0\quad  \text{a.e.},
\end{equation}
\begin{equation}
   \label{eq:J-1Ig}
J^{-1}(\R) \subset I_g(J) \cup\{\pm \infty \}
 \end{equation}
 and for $s\in  \R$,
$t\in I_g(J)$:
\begin{equation}
   \label{eq:inv2}
J(t)\leq s \Leftrightarrow t\leq J^{-1} (s) \quad\text{and}\quad
J^{-1} \circ J(t)= t.
\end{equation}
Notice that if $J$ is continuous in addition, then we have for $t \in J(\R)$:
\begin{equation} \label{eq:cont_J}
  J \circ J^{-1} (t) = t.
\end{equation}

\subsection{Cdf and copula} \label{sec:cdf_copula}
Let $X=(X_1,
\ldots, X_d)$ be a random vector on $\R^d$. Its cumulative  distribution
function  (cdf), denoted by $F$ is defined by: $F(x)=\P(X\leq x)$,
$x\in \R^d$. The corresponding one-dimensional marginals cdf are $(F_i,
1\leq i\leq d)$ with $F_i(t)=\P(X_i\leq t)$, $t\in \R$.
The cdf $F$ is called a copula if $X_i$ is uniform on $I=[0,1]$ for all $1\leq i\leq d$.
(Notice a copula is characterized by its values on $I^d$ only.)

We define $\cl_d$  the set of cdf on $\R^d$,  $\cl^{1c}_d \subset \cl_d$
the subset  of cdf  whose one-dimensional
marginals cdf  are continuous,  $\cc \subset  \cl^{1c}_d$ the  subset of
copulas. We set $\cl_d^0$ (resp. $\cc^0$) the subset of absolutely continuous
cdf (resp. copulas) on $\R^d$.

Let us define for a cdf $F$ with one-dimensional marginals $(F_i,
1\leq i\leq d)$ the function $C_F$ defined on $I^d$:
  \begin{equation} \label{eq:inversion_formula}
      C_F(y)=F(F_1^{-1}(y_1),\hdots,F_d^{-1}(y_d)), \quad y=(y_1,\hdots,y_d) \in I^d.
  \end{equation}
If $F\in \cl^{1c}_d $, then $C_F$ defined by \reff{eq:inversion_formula} is a copula thanks to \reff{eq:cont_J}. According to Sklar's theorem,
$F$ is then  completely characterized by its one-dimensional marginals cdf  $(F_i, 1\leq i\leq d)$ and the associated copula $C_F$  which
contains all information on the dependence:
\begin{equation}
\label{eq:sklar}
     F(x)=C_F\left(F_{1}(x_1),\hdots , F_{d}(x_d)\right) , \quad x=(x_1,\hdots,x_d) \in \R^d.
  \end{equation}
Equivalently,  if $X=(X_1, \ldots, X_d)$ has cdf $F$,
then $C_F$ is the cdf of  the random vector:
  \begin{equation} \label{eq:pro_int_tr}
     (F_1(X_1),\hdots,F_d(X_d)).
  \end{equation}

\subsection{Order statistics} \label{sec:order_stat}
 For $F\in \cl_d$, we write $\P_F$ the distribution of a random vector
 $X=(X_1, \hdots , X_d)$ with cdf $F$. A cdf $F\in \cl_d$ is a cdf of
 order statistics (and we shall say that $X$ is a vector of order
 statistics) if $    \P_F(X_1 \leq X_2 \leq \hdots \leq X_d)=1$.
Let us denote
by $ \cl_d^{OS} \subset \cl_d^{1c}$ the set of all cdf of order statistics with continuous
one-dimensional marginals cdf. The $d$-tuples  $(F_{i}, 1 \leq i \leq d)$ of marginal cdf's then verify $F_{i-1} \geq
F_{i}$ for all $2\leq i \leq d$ . Let $\cf_d$ be the set of $d$-tuples of continuous one-dimensional cdf's
compatible with the marginals cdf of order statistics:
\begin{equation} \label{eq:def_Fd}
\cf_d =\{\mathbf{F}=(\mathbf{F}_{i}, 1\leq i\leq d) \in \left(\cl^{1c}_1\right)^d;
\quad   \mathbf{F}_{i-1} \geq
\mathbf{F}_{i}, \, \forall 2 \leq i \leq d \}.
\end{equation}
For a given $\mathbf{F}= (\mathbf{F}_{i}, 1\leq i\leq d)$ in $\cf_d$, we define the set of cdf's $F$ of order
statistics with marginals cdf $\mathbf{F}$:
  \begin{equation} \label{eq:def_LOSF}
   \cl^{OS}_d(\mathbf{F}) = \{ F \in \cl^{OS}_d; \quad
   F_{i}=\mathbf{F}_i, \, 1\leq i\leq d\}.
  \end{equation}
  If $\mathbf{F} \in \cf_d$, then we have $ \cl^{OS}_d(\mathbf{F}) \neq \emptyset$, since the cdf of  $(\mathbf{F}_{1}^{-1}(U), \hdots, \mathbf{F}_{d}^{-1}(U))$,  $U$ uniformly distributed on $I$,  belongs to $\cl^{OS}_d(\mathbf{F})$.  We define $ \cc^{OS}(\mathbf{F})$ the set of copulas of order statistics
    with marginals $\mathbf{F}$:
\begin{equation} \label{eq:def_COSF}
    \cc^{OS}(\mathbf{F}) = \{C_F \in \cc; F \in \cl^{OS}_d(\mathbf{F}) \}.
\end{equation}
According to Sklar's  theorem, the map $F\mapsto C_F$  is a bijection between
$\cl^{OS}_d(\mathbf{F})$ and $\cc^{OS}(\mathbf{F})$ if $\mathbf{F} \in \cf_d$.

\subsection{Entropy}
  The Shannon-entropy for a
 cdf  $F \in \cl_d$ is given by:
\begin{equation}
   \label{eq:def-entropy}
H(F)= \begin{cases}
   -\infty & \text{ if } F \in \cl_d \setminus \cl_d^0, \\
-\int_{\R^d} f \log\left(f\right)  & \text{ if } F \in \cl_d^0,
\end{cases}
\end{equation}
with $f$ the density of $F$. We will use the notation $H(X)=H(F)$ if $X$
is a random vector with cdf $F$ and $H(f)=H(F)$ if $F$ has density $f$.
The entropy of any $F \in \cl^{1c}_d$  can be decomposed into the
entropy of the one-dimensional marginals cdf $(F_i, 1\leq i\leq d)$ and
the entropy of the associated copula $C_F$,  (see \cite{Zhao2011628}):
\begin{equation} \label{eq:entr_decomp}
     H(F)=\sum_{i=1} H(F_{i}) + H(C_F).
\end{equation}

For $\mathbf{F}=(\mathbf{F}_{i}, 1\leq i\leq d)  \in \cf_d$, we define
$\cj(\mathbf{F})$ taking values in $[0, +\infty ]$ by:
\begin{equation} \label{eq:cj_delta}
 \cj(\mathbf{F})=\sum_{i=2}^d \int_{\R}
 \mathbf{F}_{i}(dt)\,
 \val{\log\left(\mathbf{F}_{i-1}(t)-\mathbf{F}_{i}(t) \right)} .
\end{equation}
Our aim is
to find  the cdf  $F^* \in  \cl^{OS}_d(\mathbf{F})$ which  maximizes the
entropy  $H$.  We  shall  see that this   is  possible  if  and  only  if  $
\cj(\mathbf{F})$ is finite.   From an information theory  point of view,
this  is  the  distribution  which  is  the  least  informative  among
distributions of  order statistics with given  one-dimensional marginals
cdf $\mathbf{F}$.   Since the vector of  marginal distribution functions
$\mathbf{F}$ is  fixed, thanks to \reff{eq:entr_decomp},  we notice that
$H(F)$ is maximal on $\cl^{OS}_d(\mathbf{F})$ if and only if $H(C_F)$ is
maximal  on $\cc^{OS}(\mathbf{F})$.  Therefore we  focus on  finding the
copula $C^* \in \cc^{OS}(\mathbf{F}) $  which maximizes the entropy $H$.
We   will   give    the   solution   of   this    problem   in   Section
\ref{sec:max_entr_dens_ord_stat}  under  some additional  hypotheses  on
$\mathbf{F}$.

 \section{Symmetric copulas with given order statistics} \label{sec:connection}

   In this Section, we introduce an operator on the set $\cc^{OS}(\mathbf{F})$ of
   copulas of order statistics with fixed marginals cdf $\mathbf{F}$. This operator
   assigns to a copula $C \in \cc^{OS}(\mathbf{F})$
   the copula of the exchangeable random vector associated to the order statistics
   with marginals cdf $\mathbf{F}$ and copula $C$. We show that this operator is a bijection
   between $\cc^{OS}(\mathbf{F})$ and a set of symmetric copulas which can be characterized by their
   multidiagonal, which is a generalization of the well-known diagonal section of copulas.
   This bijection has good properties with respect to the entropy $H$, giving us
   a problem equivalent to maximizing $H$ on $\cc^{OS}(\mathbf{F})$. We shall solve this problem in
   Section \ref{sec:max_ent_copula}.

\subsection{Symmetric copulas} \label{sec:sym_copulas}
For $x=(x_1, \ldots, x_d)\in \R^d$ we define $x^{OS}=(x_{(1)}, \ldots,
x_{(d)})$ the ordered vector (increasing order) of $x$, where $x_{(1)}\leq \dots \leq
x_{(d)}$ and $\sum_{i=1} ^d \hat \delta_{x_i}=\sum_{i=1} ^d
\hat \delta_{x_{(i)}}$, with $\hat \delta_t$ the Dirac mass at $t\in \R$.

Let $\cs_d$ be the set  of permutations on $\{1,\hdots, d\}$.
For $x=(x_1, \ldots, x_d)\in \R^d$ and $\pi\in \cs_d$, we set
$x_\pi =(x_{\pi(1)},\hdots,x_{\pi(d)})$.
A function $h$ defined on $\R^d$ is symmetric if $h(x_\pi)=h(x)$ for all
$\pi\in  \cs_d$.  A  random  vector  $X$  taking  values  in  $\R^d$  is
exchangeable if $X_\pi$ is distributed as  $X$ for all $\pi\in
\cs_d$. In particular a
random vector $X$ taking values in $\R^d$ is exchangeable if and only if
its cdf is symmetric.  Let  $\cl_d^{sym}$ (resp. $\cc^{sym}$) denote the
subset of $\cl_d$ (resp. $\cc$) of symmetric cdf (resp. copulas) on $\R^d$.

Let $F\in
\cl_d^{1c}$ and define its symmetrization $F^{sym}\in \cl_d^{sym}$ by:
\begin{equation}
   \label{eq:Fsym}
  F^{sym}(x)=\inv{d!}\sum_ {\pi\in \cs_d} F(x_\pi), \quad x\in \R^d.
\end{equation}
In particular, if $X$ is a  random  vector  taking  values  in  $\R^d$
with cdf $F$  and $\Pi$ is a random variable independent of $X$,
 uniformly distributed on $\cs_d$, then $X_\Pi$ is exchangeable with
cdf $F^{sym}$.

We define the following operator on the set of copulas of order
statistics.
\begin{defi} \label{defi:SF}
  Let $  \mathbf{F} \in \left(\cl^{1c}_1\right)^d$.  For $C\in \cc  $ we
  define  $S_\mathbf{F}(C)$ as  the  copula of  the exchangeable  random
  variable   $X_\Pi$,   where   $X$   is  a random vector on $\R^d$   with
  one-dimensional marginals cdf $ \mathbf{F}  $ and copula $C$ and $\Pi$
  is an independent random variable  uniform on $\cs_d$.
\end{defi}
The application
  $S_\mathbf{F}$ is well-defined on $\cc $ and takes values in $\cc^{sym}$.
In the above definition,  with $ \mathbf{F}=(\mathbf{F}_i, 1\leq i\leq
d)$, the one-dimensional marginals cdf  of $X_\Pi$  are equal to:
\begin{equation}\label{eq:def_G}
     G=\inv{d} \sum_{i=1}^d \mathbf{F}_i.
\end{equation}
Since the one-dimensional marginals  cdf $ \mathbf{F}_i$ are continuous,
we get that $G$  is continuous and thus the cdf  of $X_\Pi$ belongs
to $\cl_d^{1c}$. In particular, thanks to Sklar's theorem, the copula of
$X_ \Pi$ is indeed uniquely defined.

Combining      \reff{eq:sklar},     \reff{eq:inversion_formula}      and
\reff{eq:Fsym}, we can give an explicit formula for $S_\mathbf{F}(C)$:
      \begin{equation} \label{eq:SFC:expl}
         S_\mathbf{F}(C)(u)= \inv{d!} \sum_{\pi \in \cs_d}
         C\left(\mathbf{F}_1(G^{-1}(u_{\pi(1)})), \hdots,
           \mathbf{F}_d(G^{-1}(u_{\pi(d)}))\right), \quad u\in I^d.
      \end{equation}

\begin{rem}
  The copula  $S_\mathbf{F}(C)$ is  not  equal in general to  the
  exchangeable copula $C^{sym}$ defined similarly to \reff{eq:Fsym} by
  $C^{sym}=(1/d!)\sum_{\pi \in \cs_d} C(x_\pi)$. However this is the case
  if the one-dimensional marginals cdf $\mathbf{F}_i$ are all equal, in
  which case  $\mathbf{F}_i=G$ for all $1 \leq i \leq d$.
\end{rem}

If $X$  is a  random vector  on $\R^d$,  let $X^{OS}  = (X_{(1)},\hdots,
X_{(d)})$ be  the order  statistics of  $X$. The
proof of the next Lemma is elementary.

    \begin{lem} \label{lem:exch_OS} Let $X$ be a random vector on $\R^d$
      with cdf  $F\in \cl_d$  and $\Pi$  a  random variable independent of $X$,
      uniformly distributed on $\cs_d$. We have:
    \begin{itemize}
     \item If $F \in \cl_{d}^{OS}$, then a.s. $(X_\Pi)^{OS}=X$
     \item If $F \in \cl_{d}^{sym}$, then $(X^{OS})_\Pi$ has the same distribution as $X$.
    \end{itemize}
   \end{lem}

For   $\mathbf{F}  \in \cf_d$, we define  the set  of
copulas  $\cc^{sym}(\mathbf{F})  \subset  \cc^{sym}$ as  the  image  of
$\cc^{OS}(\mathbf{F})$ by the symmetrizing operator $S_\mathbf{F}$:
\begin{equation}
   \label{eq:def-CFsym}
    \cc^{sym}(\mathbf{F}) = S_{\mathbf{F}}(\cc^{OS}(\mathbf{F})).
\end{equation}

The following Lemma is one of the main result of this section.

\begin{lem}
\label{lem:one_to_one}
Let $\mathbf{F} \in \cf_d$.  The symmetrizing operator $S_{\mathbf{F}}$ is
a  bijection  from $\cc^{OS}(\mathbf{F})$  onto
$\cc^{sym}(\mathbf{F})$.
\end{lem}

\begin{proof}
  Let       $C_1,       C_2       \in       \cc^{OS}(\mathbf{F})$       with
  $S_\mathbf{F}(C_1)=S_\mathbf{F}(C_2)$. Let  $X$ and $Y$ be   random
  vectors   with    one-dimensional marginals cdf   $\mathbf{F}$   and    copula   $C_1,C_2$
  respectively. Since $C_1,  C_2 \in \cc^{OS}(\mathbf{F})$, we  get that $X$
  and $Y$ are order statistics.  Notice $X_\Pi$ and $Y_\Pi$ have the same
  one-dimensional marginals according to \reff{eq:def_G} and same copula
  given by $S_\mathbf{F}(C_1)=S_\mathbf{F}(C_2)$.  Therefore $X_\Pi$ and
  $Y_\Pi$ have  the same  distribution. Thus, their  corresponding order
  statistics   $(X_\Pi)^{OS}$   and   $(Y_\Pi)^{OS}$   have   the   same
  distribution. By Lemma \ref{lem:exch_OS} we  get that $X$ and $Y$ have
  the same distribution as well, which implies $C_1=C_2$.
\end{proof}

\begin{rem}
  We have in general  $\cc^{sym}(\mathbf{F}) \neq \cc^{OS}(\mathbf{F}) \cap
  \cc^{sym}$.  One  exception  being  when  the  marginals  cdf's  ${\mathbf{F}}_{i}$ are all equal. In this case, both sides
  reduce to one  copula which is the  Fr\'{e}chet-Hoeffding upper bound
  copula: $C^+(u) = \min u$, $u\in I^d$.
\end{rem}

\subsection{Multidiagonals and characterization of $\cc^{sym}(\mathbf{F})$} \label{sec:multidiag}

Let $C\in \cc$ be  a copula and $U$ a random vector  with cdf $C$. The
map $t\mapsto C(t,\hdots,t)$ for $t \in I$, which is called the
diagonal section of $C$,  is the cdf of $\max U$. We shall consider a
generalization of the diagonal section of $C$ in the next Definition.

\begin{defi}
   \label{defi:diag}
Let $C\in \cc$ be  a copula on $\R^d$ and $U$ a random vector  with cdf
$C$. The
  \textit{multidiagonal} of the copula $C$,
  $\delta_C=(\delta_{(i)},1\leq i\leq d)$, is the $d$-tuple of
  the one-dimensional marginals cdf of $U^{OS}=(U_{(1)},\hdots, U_{(d)})$ the order statistics of
$U$: for $1\leq i\leq d$
\[
\delta_{(i)}(t)=\P(U_{(i)}\leq t), \quad t\in I.
\]
\end{defi}

We denote by $\cd=\{\delta_C; C \in
  \cc\}$ the set of multidiagonals. Notice that $\cd \subset \cf_d$, see Remark \ref{rem:delta_Lipshitz}.
For $\delta\in \cd$ a multidiagonal, we define  $\cc_\delta=\{C;
\delta_C = \delta\}$ the set of copulas with multidiagonal
$\delta$.

A characterization of the set $\cd$ is given by Theorem 1 of
  \cite{jaworski2008distributions}: a vector of functions
  $\delta=(\delta_{(1)},\hdots, \delta_{(d)})$ belongs to $\cd$ if and
  only if $\delta_{(i)} \in \cl_1$ and the following conditions hold:
 \begin{equation} \label{eq:d_cond1}
  \delta_{(i)}\geq \delta_{(i+1)}, \quad  1\leq i \leq d-1,
  \end{equation}
  \begin{equation} \label{eq:sum_delta}
  \sum_{i=1}^d \delta_{(i)}(s) = ds, \quad \quad 0 \leq s \leq 1.
 \end{equation}

 \begin{rem}       \label{rem:delta_Lipshitz}        The       condition
   \reff{eq:sum_delta}  implies that  $\delta_{(i)}  \in \cl_1^{1c}$,  $
   1\leq i \leq d$, moreover they  are $d$-Lipschitz. Also, it is enough
   to know $d-1$  functions from $\delta_{(i)}$, $1 \leq i  \leq d$, the
   remaining one is implicitly defined by \reff{eq:sum_delta}. Condition
   \reff{eq:d_cond1} along with the continuity of $\delta_{(i)}$ implies
   that any multidiagonal $\delta_{C}$ is compatible with the continuous
   marginal distributions of an order statistics, therefore $\cd \subset
   \cf_d$.
    \end{rem}

  \begin{rem} \label{rem:dlogd}
   Since $\delta_{(i)}$, $1 \leq i \leq d$ are non-decreasing
   and $d$-Lipschitz, we have for almost every $t\in I$: $0 \leq
      (\delta_{(i)})'(t) \leq d$ and thus
      $\val{(\delta_{(i)})'(t)
        \log((\delta_{(i)})'(t))} \leq d\log(d)$ for $d \geq 2$. We deduce
        that for $d \geq 2$:
    \begin{equation} \label{eq:delta_bound}
     \val{H(\delta_{(i)})} \leq d\log(d).
    \end{equation}
  \end{rem}

    \begin{rem}
Let $C\in \cc^{sym}$ be  a symmetric copula on $\R^d$ and $U$ a random vector  with cdf
$C$. We  check that the multidiagonal    $\delta_C=(\delta_{(i)},1\leq i\leq d)$
can be  expressed in terms of  the diagonal sections  $(C_{\{i\}}, 1\leq i\leq d )$ where  for $1
      \leq i \leq d $:
  \[
C_{\{i\}}(t)=\P\left(\max_{1\leq k \leq i} {U}_k \leq t \right)
={C}(\underbrace{t,\hdots,t}_{i \text{ terms }},
\underbrace{1,\hdots,1}_{d-i \text{ terms}}), \quad t\in I.
  \]
According to 2.8 of \cite{jaworski2008distributions}, we have for $1\leq i\leq d$:
  \[
   \delta_{(i)}(t)=\sum_{j=i}^d (-1)^{j-i}  \binom{j-1}{i-1}
   \binom{d}{j} {C}_{\{j\}}(t), \quad t\in I.
  \]
  Conversely, we can express the functions $(C_{\{i\}}, 1\leq i\leq d )$
with $\delta_C$. For $1\leq i\leq d$ and $t\in I$, we have:
  \begin{align*}
    C_{\{i\}}(t) & = \P\left(\max_{1\leq k \leq i} {U}_k \leq t \right) \\
                         & = \sum_{j=i}^d \P\left({U}_{(j)}\leq t \big|
                         \max_{1\leq k \leq i} {U}_k = {U}_{(j)} \right)
                         \P\left(\max_{1\leq k \leq i} {U}_k = {U}_{(j)}\right) \\
                         & = \sum_{j=i}^d \frac{\binom{j-1}{i-1}}{\binom{d}{i}}\delta_{(j)}(t) ,
  \end{align*}
  where we used the definition of $\delta_{(i)}$ and
  the exchangeability of ${U}$ for the third equality.
 \end{rem}

 The next  technical Lemma  will be used  in forthcoming  proofs. Recall
 that  $J^{-1}$  denotes the  generalized  inverse  of a  non-decreasing
 function $J$, see Section \ref{sec:not-J-1} for its definition and
 properties, in particular,  $J^{-1} \circ J (t)\leq t$ for $t\in \R$.  Recall also
 that  for  $x=(x_1,  \ldots,  x_d)\in \R^d$,  we  write  $G(x)=(G(x_1),
 \ldots, G(x_d))$.

   \begin{lem} \label{lem:G_Ginv}
     Let $X=(X_1, \ldots, X_d)$ be a random vector on $\R^d$ with
     one-dimensional margi\-nals cdf $(F_i, 1\leq i\leq d)\in (\cl_1)^d$. Set $G=
     \sum_{i=1}^d F_i/d$. We have for $1 \leq i \leq d$:
     \begin{equation} \label{eq:FiG-1G=Fi}
       \P(X_i\leq G^{-1} \circ G(t))=\P(X_i\leq t), \quad t\in \R, \quad\text{that is}\quad
   F_i\circ G^{-1}\circ G =F_i.
     \end{equation}
   We also have for $x\in \R^d$:
     \begin{equation} \label{eq:G_Ginv_all}
       \P(G(X) \leq x) =\P(X\leq
       G^{-1}(x) ).
     \end{equation}
   \end{lem}

    \begin{proof}
     Since $G$ is the average of the non-decreasing functions $F_i$, if
     $G(s)=G(s')$ for some $s,s'\in \R$, then we have $F_i(s)=F_i(s')$ for every $1
     \leq i \leq d$. Thanks to \reff{eq:inv}, we have
     $G \circ G^{-1} \circ G(t)=G(t)$ and thus  $F_i \circ G^{-1} \circ G(t)=F_i(t)$. This gives
     \reff{eq:FiG-1G=Fi}.

     Recall definition \reff{eq:def-Ig} for $I_g(J)$ the set of points
     where  the function $J$ is increasing on their left. Since $G$ is the average of
     the non-decreasing functions $F_i$, we deduce that $
     I_g(G)=\bigcup _{1\leq i\leq d} I_g(F_i)$. Notice that a.s. $X_i$
     belongs to $I_g(F_i)$. Thanks to \reff{eq:inv2}, we get that
     a.s. $\{G(X)\leq x\}=\{X\leq G^{-1}(x)\}$. This gives
     \reff{eq:G_Ginv_all}.
    \end{proof}

    We will also require the following Lemma.

    \begin{lem} \label{lem:FiG-1}
    Let $X=(X_1, \ldots, X_d)$ be a random vector on $\R^d$ with
     one-dimensional margi\-nals cdf $(F_i, 1\leq i\leq d)\in (\cl_1)^d$. Set $G=
     \sum_{i=1}^d F_i/d$. We have for $1 \leq i \leq d$:
     \begin{equation} \label{eq:FiG-1}
       (F_i \circ G^{-1})^{-1} = G \circ F_i^{-1}.
     \end{equation}
    \end{lem}

    \begin{proof}

    Recall Definition \reff{eq:def-Ig} for $I_g(J)$ the set of points
     where  the function $J$ is increasing on their left. Let $1\leq
     i\leq d$.
Thanks to  \reff{eq:J-1Ig}, we have $F^{-1}_{i}(\R) \subset
I_g(F_{i})\cup\{\pm \infty \}$. Since $G$ is the average of
     the non-decreasing functions $F_i$, we deduce that $
     I_g(G)=\bigcup _{1\leq i\leq d} I_g(F_i)$. Thus we get:
\begin{equation}
   \label{eq:inclusion}
F^{-1}_{i}(\R) \subset I_g(G)\cup\{\pm \infty \},
\end{equation}
for all $1 \leq i \leq d$. The function $F_i \circ G^{-1}$ is also bounded, non-decreasing and c\`{a}d-l\`{a}g therefore we have for $t,s, \in \R$:
\[
t \geq (F_{i}\circ G^{-1})^{-1} (s)
\Longleftrightarrow
F_{i}\circ G^{-1}(t) \geq  s
\Longleftrightarrow
G^{-1}(t) \geq F^{-1}_{i}(s)
\Longleftrightarrow
t \geq G\circ F^{-1}_{i}(s),
\]
where we used  the equivalence of
\reff{eq:inv} for the first and second equivalence, \reff{eq:inclusion} and the equivalence of
\reff{eq:inv2} for the last. This gives that  $(F_i \circ G^{-1})^{-1}=G\circ
F^{-1}_{i}$.

    \end{proof}

    In the following Lemma, we show that for $\mathbf{F} \in \cf_d$, all copulas
    in $\cc^{sym}(\mathbf{F})$ share the same multidiagonal denoted by
    $\delta^{\mathbf{F}}$.

    \begin{lem}     \label{lem:delta_SFC}      Let     $\mathbf{F}     =
      (\mathbf{F}_{i}, 1\leq i\leq d)  \in \cf_d$. Let  $C  \in
      \cc^{OS}(\mathbf{F})$ and  $U$ be  a random vector with  cdf
      $S_{\mathbf{F}}(C)$.       Let   $\delta^{\mathbf{F}}=
      (\delta_{(i)},1\leq i\leq  d)$ be the multidiagonal of $S_{\mathbf{F}}(C)$,  that is the
      one-dimensional marginals cdf  of $U^{OS}$, the order statistics of
      $U$. We have that
      $\delta^{\mathbf{F}}$ does not depend on $C$ and for $ 1
      \leq i \leq d$:
\begin{equation}
   \label{eq:d_i+-1}
       \delta_{(i)}= \mathbf{F}_i\circ G^{-1} \quad\text{and}\quad
      \delta_{(i)}^{-1}= G\circ \mathbf{F}_i^{-1},
\end{equation}
    with $G$ given by \reff{eq:def_G}. Furthermore, $C$ is the unique
    copula of $U^{OS}$.
    \end{lem}
With obvious notation, we might simply write $\delta^{\mathbf{F}}=
\mathbf{F}\circ G^{-1}$, with $G$ given by \reff{eq:def_G}.

    \begin{proof}
      Let $X$  be a random  vector of order statistics  with marginals
      $\mathbf{F} \in \cf_d$ and  copula $C$.  Then $S_{\mathbf{F}}(C)$ is
      the copula of the exchangeable  random vector $X_\Pi$, where $\Pi$
      is uniform  on $\cs_d$  and independent of  $X$.  We  have already
      seen  in \reff{eq:def_G}  that  the  one-dimensional marginals  of
      $X_\Pi$ have  the same distribution  given by $G  \in \cl_1^{1c}$.
      Thanks to  \reff{eq:pro_int_tr}, we deduce that  the random vector
      $U$, with  cdf $S_{\mathbf{F}}(C)$,  has the same  distribution as
      $G(X_{\Pi})$. Since  $G$ is non-decreasing, this  implies that the
      order statistics  of $U$,  $U^{OS}$, has  the same  distribution as
      $G\left((X_\Pi)^{OS}\right)$ that  is as  $G(X)$, thanks  to Lemma
      \ref{lem:exch_OS}. Then use \reff{eq:G_Ginv_all} to get for $x\in \R^d$:
\begin{equation}\label{eq:U^OS}
\P(U^{OS}\leq x)=\P(G(X)\leq x)=\P(X\leq G^{-1}(x)).
\end{equation}
This gives the first part of the Lemma as the multidiagonal of $U$ is
the one-dimensional marginals cdf of its order statistics.
The second equation in \reff{eq:d_i+-1} is due to Lemma \ref{lem:FiG-1}.
The fact that $C$ is the copula of $U^{OS}$ and its uniqueness are due to \reff{eq:U^OS}
and the continuity of $\delta_{(i)}$, see Remark \ref{rem:delta_Lipshitz}.
    \end{proof}

    According  to the  next Proposition,  \reff{eq:def-CFsym} and  Lemma
    \ref{lem:one_to_one}, we  get the main  result of this  Section: for
    any    $\mathbf{F}   \in    \cf_d$,   the    symmetrizing   operator
    $S_{\mathbf{F}}$ is  a bijection between  $\cc^{OS}(\mathbf{F})$ and
    $\cc_{\delta^{\mathbf{F}}} \cap \cc^{sym}$.

     \begin{prop}\label{prop:char_cfsym}
       Let $\mathbf{F} \in \cf_d$. We have $\cc^{sym}(\mathbf{F}) = \cc_{\delta^{\mathbf{F}}} \cap \cc^{sym}$.
     \end{prop}

     \begin{proof}
       By Lemma \ref{lem:delta_SFC}, we have $\cc^{sym}(\mathbf{F})
       \subset \cc_{\delta^{\mathbf{F}}} \cap \cc^{sym}$.

       Let $C\in \cc_{\delta^{\mathbf{F}}} \cap  \cc^{sym}$ and $U$ be a
       random vector with cdf $C$.  Let $G$ be given by \reff{eq:def_G}.
       Notice that  $X=G^{-1}(U)$ is an exchangeable  random vector with
       marginals $G$ and copula $C$.  Thanks to Lemma \ref{lem:exch_OS},
       the  proof  will  be  complete  as soon  as  we  prove  that  the
       one-dimensional  marginals   cdf  of   $X^{OS}=(X_{(1)},  \ldots,
       X_{(d)})$,  the order  statistics of  $X$, is  given by  $\mathbf
       {F}$.  Notice $X^{OS}=G^{-1}(U^{OS})$,  with  $U^{OS}$ the  order
       statistics of  $U$ whose one-dimensional marginals  cdf are given
       by $\delta^{\mathbf {F}}$.  We have for $1\leq i\leq d$ and $t\in
       \R$:
\[
\P(X_{(i)}\leq t ) = \P(G^{-1}(U_{(i)})\leq  t)
=\P(U_{(i)}\leq G(t))={\mathbf F}_i\circ
G^{-1}\circ G(t)= {\mathbf F}_i(t),
\]
where we used \reff{eq:inv} for the second equality, \reff{eq:d_i+-1}
for the third, and \reff{eq:FiG-1G=Fi} for the last. This finishes the proof.
\end{proof}

We end this Section by an ancillary result we shall use later.

\begin{lem}
   \label{lem:J(F)}
   Let $\mathbf{F} \in \cf_d$. We have
$\cj(\mathbf{F})=\cj(\delta^\mathbf{F})$.
\end{lem}

\begin{proof}
Let $\mathbf{F} =(\mathbf{F}_i, 1\leq i\leq d)$.    We get,
using \reff{eq:d_i+-1} and the change of variable $s=G^{-1}(t)$ that:
\[
\int_I \delta_{(i)}'(t)
 \val{\log\left(\delta_{(i-1)}(t)-\delta_{(i)}(t) \right)} \, dt
= \int_{G^{-1}((0,1))} \mathbf{F}_i(ds) \,
\val{\log\left(\mathbf{F}_{i-1}(s) - \mathbf{F}_{i}(s) \right)}.
\]
Since $d\mathbf{F}_i=0$ outside $G^{-1}((0,1))$ (as $G$ is increasing as
soon as $\mathbf{F}_{i}$ is increasing), we get the above last
integration is also over $\R$. We deduce that:
\[
\cj(\delta^\mathbf{F})=\sum_{i=2}^d \int_\R \mathbf{F}_i(ds) \,
\val{\log\left(\mathbf{F}_{i-1}(s) - \mathbf{F}_{i}(s) \right)}
=\cj(\mathbf{F}).
\]
\end{proof}

\subsection{Density and entropy of copulas in $\cc^{sym}(\mathbf{F})$} \label{sec:dens+entropy}

We prove in this Section that $S_{\mathbf{F}}$ preserves the absolute
continuity on $\cc^{OS}(\mathbf{F})$ for $\mathbf{F}\in \cf_d$ and the entropy
up to a constant.  Let us introduce some notation. For marginals
$\mathbf{F} \in \cf_d$, let
\begin{equation} \label{eq:def_Psi_i_F}
   \Psi^\mathbf{F}_i = \{ s \in \R, \, \mathbf{F}_{i-1}(s) > \mathbf{F}_{i}(s) \} \quad \text{ for } 2 \leq i \leq d.
\end{equation}
The complementary set $(\Psi^\mathbf{F}_i)^c$ is the collection of the points where $\mathbf{F}_{i-1}=\mathbf{F}_{i}$.
We define $\Sigma^\mathbf{F} \subset I $ as:
\begin{equation} \label{eq:def_sigma_F}
\Sigma^\mathbf{F} = \bigcup_{i=2}^d \mathbf{F}_{i}\left((\Psi^\mathbf{F}_i)^c\right).
\end{equation}
By  Remark \ref{rem:delta_Lipshitz}  we have $\cd  \subset \cf_d$, then
the  definitions \reff{eq:def_Psi_i_F}  and \reff{eq:def_sigma_F}  apply
for  all $\delta  \in  \cd$. In  particular, for  $\delta=(\delta_{(1)},
\hdots, \delta_{(d)}) \in \cd$ the  sets $\Psi^\delta_i$, $2 \leq i \leq
d$ are  open subsets of $I$,  therefore $(\Psi^\delta_i)^c \cap I$  is a
compact subset.  This and  the continuity  of $\delta_{(i)}$  imply that
$\delta_{(i)}((\Psi^\delta_i)^c)=\delta_{(i)}((\Psi^\delta_i)^c \cap I)$
is also compact, hence $\Sigma^\delta$  is compact. Notice that $\{0,1\}
\subset \Sigma^\delta$ always holds.  We define $\cc_\delta^0=\cc_\delta
\cap   \cc^0$  the   subset  of   absolutely  continuous   copulas  with
multidiagonal   $\delta$  and   the  subset   $\cd^0=\{\delta\in  \cd,\,
\cc_\delta^0  \neq   \emptyset\}  $  of  multidiagonals   of  absolutely
continuous     copulas.       According     to     Theorem      2     of
\cite{jaworski2008distributions}, the multidiagonal  $\delta$ belongs to
$\cd^0$  if and  only if  it belongs  to $\cd$ and the Lebesgue  measure of
$\Sigma^\delta$ is zero: $\val{ \Sigma^\delta}=0$.

\begin{lem} \label{lem:delta'_Psi}
 Let $\delta \in \cd$.  We have $\delta\in \cd^0$ if and only if  for all
  $2 \leq i \leq d$, a.e.:
 \begin{equation} \label{eq:delta'=0}
   \delta'_{(i-1)} \ind_{(\Psi_i^\delta)^c} = \delta'_{(i)} \ind_{(\Psi_i^\delta)^c} =0.
 \end{equation}
Furthermore, we have that $\cj(\delta)<+\infty $ implies  $\delta\in
\cd^0$.
\end{lem}

\begin{proof}
Let $J$ be a function defined on $I$,  Lipschitz and non-decreasing. Let $A$ be a
Borel subset of $I$. We have:
\[
|J(A)|=\int\ind_{J(A)} (t)\, dt= \int_0^1 \ind_{\{s\in J^{-1}\circ J(A)\}} \,
J'(s)ds= \int_0^1 \ind_{A}(s) \,
J'(s)ds,
\]
where we used \reff{eq:IgJc} and \reff{eq:inv2} for the last
equality. This gives that  $|J(A)|=0$ if and only if
a.e. $J'\ind_A=0$. Then use that $\delta\in \cd^0$ if and only if
$\val{ \delta_{(i)}((\Psi_i^\delta)^c)}=0$ for all $1\leq i\leq d$ and
that ${\delta_{(i-1)}((\Psi_i^\delta)^c)} =
{ \delta_{(i)}((\Psi_i^\delta)^c)}$ to conclude that $\delta\in \cd^0$
if and only if  \reff{eq:delta'=0} holds for all $2\leq i\leq d$.  The
last part of the
Lemma is clear.
\end{proof}

\begin{defi}
   \label{defi:f_d^0}
Let $\cf_d^0 \subset \cf_d$ be the subset of marginals $\mathbf{F}$ such
that there exists an absolutely continuous cdf of order statistics with
marginals $\mathbf{F}$, that is $\cl_d^{OS }(\mathbf{F}) \cap \cl_d^0
\neq \emptyset$.
\end{defi}

In particular, we have $\cd^0 \subset \cf_d^0$.
The next Lemma gives a characterization
of the set $\cf_d^0$.

\begin{lem} \label{lem:F^0} Let $\mathbf{F} \in \cf_d$. Then $\mathbf{F}
  \in \cf_d^0$ if  and only if $\mathbf{F}_i \in \cl_1^0$  for $1 \leq i
  \leq d$  and $\val{\Sigma^\mathbf{F}}=0$.   Furthermore, we  have that
  $\mathbf{F}_i   \in   \cl_1^0$   for   $1   \leq   i   \leq   d$   and
  $\cj(\mathbf{F})<+\infty $ imply $\mathbf{F} \in \cf_d^0$.
\end{lem}

\begin{proof}
  Let $F \in  \cl_d^{OS }(\mathbf{F})$. We know that $F  \in \cl_d^0$ if
  and only if $\mathbf{F}_i \in \cl_1^0$ for  $1 \leq i \leq d$ and $C_F
  \in  \cc^0$, the  subset  of absolutely  continuous  copulas (see  for
  example \cite{Jaworski20092863}).  Therefore $\mathbf{F}  \in \cf_d^0$
  if and only  if $\mathbf{F}_i \in \cl_1^0$  for $1 \leq i  \leq d$ and
  $\cc^{OS}(\mathbf{F})   \cap  \cc^0   \neq  \emptyset$.   Recall  that
  $\delta^{\mathbf{F}}$ is  defined by  \reff{eq:d_i+-1}. We  first show
  that
 \begin{equation} \label{eq:lem_F^0_iff}
    \cc^{OS}(\mathbf{F}) \cap \cc^0 \neq \emptyset \text{ if and only if } \cc^0_{\delta^\mathbf{F}} \cap \cc^{sym} \neq \emptyset.
 \end{equation}

 Let   $C    \in   \cc^{OS}(\mathbf{F})   \cap   \cc^0$.    Then   Lemma
 \ref{lem:delta_SFC}     ensures     that     $S_{\mathbf{F}}(C)     \in
 \cc_{\delta^\mathbf{F}}  \cap \cc^{sym}$.  The  absolute continuity  of
 $S_{\mathbf{F}}(C)$  is  a  direct consequence  of  \reff{eq:SFC:expl},
 \reff{eq:d_i+-1} and Remark \ref{rem:delta_Lipshitz} which ensures that
 $\delta_{(i)}^{\mathbf{F}},   1\leq  i   \leq  d$   are  $d$-Lipschitz,
 therefore  their  derivatives exist  a.e.  on  $I$. This  ensures  that
 $\cc^0_{\delta^\mathbf{F}} \cap \cc^{sym} \neq \emptyset$.

 Conversely, let  $C \in \cc^0_{\delta^\mathbf{F}} \cap  \cc^{sym}$. Let
 $U$ be a random vector with cdf $C$. Then its order statistics $U^{OS}$
 is also absolutely continuous. Therefore  the copula of $U^{OS}$, which
 is  $S_{\mathbf{F}}^{-1}(C)$  by  Lemma  \ref{lem:delta_SFC},  is  also
 absolutely   continuous.     This   proves   thanks    to   Proposition
 \ref{prop:char_cfsym}     and    Lemma     \ref{lem:one_to_one}    that
 $S_{\mathbf{F}}^{-1}(C)  \in  \cc^{OS}(\mathbf{F}) \cap  \cc^0$.   This
 gives \reff{eq:lem_F^0_iff}.

 Notice that  $\cc^0_{\delta^\mathbf{F}} \cap \cc^{sym}  \neq \emptyset$
 is equivalent  to $\cc^0_{\delta^\mathbf{F}}\neq \emptyset$,  since for
 any $C \in \cc^0_{\delta^\mathbf{F}}$ we have that $C^{sym}$ defined by
 \reff{eq:Fsym} belongs  to $\cc^0_{\delta^\mathbf{F}}  \cap \cc^{sym}$.
 By       Theorem      2       of      \cite{jaworski2008distributions},
 $\cc^0_{\delta^\mathbf{F}}   \neq    \emptyset$   if   and    only   if
 $\Sigma^{\delta^\mathbf{F}}$ has  zero Lebesgue  measure. The  proof is
 then  complete as  one can  easily verify  using \reff{eq:d_i+-1}  that
 $\Sigma^{\delta^\mathbf{F}}  = \Sigma^\mathbf{F}$  and thanks  to Lemma
 \ref{lem:J(F)}.
 \end{proof}

 From now on we consider $\mathbf{F} \in \cf_d^0$. We give an auxiliary lemma on the support of the copulas in $\cc^{OS}(\mathbf{F}) \cap \cc^0$.

   \begin{lem} \label{lem:c_support}
    Let $\mathbf{F}=(\mathbf{F}_i, 1\leq i\leq d) \in \cf_d^0$ and $C \in
    \cc^{OS}(\mathbf{F}) \cap \cc^0$. Then the density of $C$ vanishes a.e. on
    $I^d\setminus T^{\mathbf{F}}$ with:
\begin{equation} \label{eq:c_zero-T}
T^{\mathbf{F}}=\{ u=(u_1, \ldots, u_d)\in I^d; \, \mathbf{F}_1^{-1}(u_1)\leq \dots\leq
\mathbf{F}_d^{-1} (u_d)\}.
\end{equation}
   \end{lem}

    \begin{proof}
     Let $X=(X_1,\hdots,X_d)$ be a random vector of order
     statistics with one-dimensional marginals cdf $\mathbf{F}$ and
     copula $C \in \cc^0$. Let  $U=(U_1,\hdots, U_d) $ be a random vector with cdf
     $C$. Then it is distributed as $(\mathbf{F}_1(X_1),\hdots,
     \mathbf{F}_d(X_d))$,   see
     \reff{eq:pro_int_tr}. We get $\P( U \in T^{\mathbf{F}})=1$, since
     $X$ is a vector of order statistics and $X_i \in I_g(\mathbf{F}_i)$ a.s. for $1 \leq i \leq d$. This gives the result.
    \end{proof}

Now we establish the connection between the sets $\cc^{OS}(\mathbf{F}) \cap \cc^0$ and
$\cc^{sym}(\mathbf{F}) \cap \cc^0$.

   \begin{lem} \label{lem:dens_c_ctilde}
    Let $\mathbf{F} \in \cf_d^0$. The symmetrizing operator $S_{\mathbf{F}}$
    is a bijection from $\cc^{OS}(\mathbf{F}) \cap \cc^0$ onto
    $\cc^{sym}(\mathbf{F}) \cap \cc^0$. Moreover, if $C \in
    \cc^{OS}(\mathbf{F}) \cap \cc^0$, with density function $c$, then the density
    function $s_{\mathbf{F}}(C)$ of $S_{\mathbf{F}}(C)$ is given by, for a.e. $u=(u_1, \ldots, u_d)\in I^d$:
    \begin{equation} \label{eq:c_to_ctilde}
        s_{\mathbf{F}}(C)(u) = \inv{d!}
        c\left(\delta_{(1)}(u_{(1)}) , \hdots,
          \delta_{(d)}(u_{(d)}) \right) \prod_{i=1}^d
        \delta_{(i)}'(u_{(i)}).
    \end{equation}
    Let $T^{\mathbf{F}}$ be given by \reff{eq:c_zero-T}.
    If $C \in \cc^{sym}(\mathbf{F}) \cap \cc_0$ with density ${c}$,
    then the density $s_{\mathbf{F}}^{-1}(C)$ of
    $S_{\mathbf{F}}^{-1}({C})$ is given by, for a.e. $u=(u_1, \ldots, u_d)\in I^d$:
    \begin{equation} \label{eq:c_tilde_to_c}
     s_{\mathbf{F}}^{-1}(C)(u)= d!\,
     \frac{{c}\left(\delta_{(1)}^{-1}(u_1), \hdots,
         \delta_{(d)}^{-1}(u_d) \right)}{\prod_{i=1}^d
       \delta_{(i)}'\circ \delta_{(i)}^{-1}(u_i)
       } \ind_{T^{\mathbf{F}}}(u) \ind_{\left\{\prod_{i=1}^d \delta_{(i)}'\circ \delta_{(i)}^{-1}(u_i)
     > 0\right\}}.
    \end{equation}
   \end{lem}

    \begin{proof}
      By Proposition \ref{prop:char_cfsym}, we deduce that $\cc^{sym}(\mathbf{F}) \cap \cc^0 =
      \cc^0_{\delta^\mathbf{F}} \cap \cc^{sym}$. Lemma \ref{lem:one_to_one} and the  proof of Lemma
      \ref{lem:F^0} ensures that $S_{\mathbf{F}}$ is a bijection between $\cc^{OS}(\mathbf{F}) \cap \cc^0$ and
    $\cc^{sym}(\mathbf{F}) \cap \cc^0$.
      The explicit formula
     \reff{eq:c_to_ctilde} can be obtained by taking the mixed
     derivative of the right hand side of \reff{eq:SFC:expl}. By Lemma
     \ref{lem:c_support}, all the terms in the sum disappear except the
     one on the right hand side of \reff{eq:c_to_ctilde}.

To obtain \reff{eq:c_tilde_to_c}, let ${C} \in \cc^{sym}(\mathbf{F}) \cap\, \cc^0$ with density $c$, and  $U$ be a
random vector with cdf  ${C}$. The order statistics $U^{OS}$ derived
from $U$ is also absolutely continuous with cumulative distribution
function ${K}$, and density function ${k}$ given by:
     \begin{equation*}
        {k}(u) = d! \, {c}(u) \ind_{\triangle}(u), \quad u\in I
        ^d.
     \end{equation*}
By Lemma \ref{lem:delta_SFC}, $S_{\mathbf{F}}^{-1}({C})$ is the copula of
$U^{OS}$. From \reff{eq:inversion_formula}, we have for
$u=(u_1, \ldots, u_d)\in I^d$:
     \begin{equation}  \label{eq:kOS}
        S_{\mathbf{F}}^{-1}({C})(u) =
        {K}(\delta_{(1)}^{-1}(u_1), \hdots,
        \delta_{(d)}^{-1}(u_d)).
     \end{equation}
According to  \reff{eq:inv2}, we deduce that
$G^{-1}\circ G \circ \mathbf{F}_i^{-1}=\mathbf{F}_i^{-1}$ on
$(0,1)$. This implies that for $s,t\in (0,1)$, $1\leq i<j\leq  d$:
\begin{align*}
\delta_{(i)}^{-1}(s)\leq  \delta_{(j)}^{-1}(t)
&\Leftrightarrow
G\circ \mathbf{F}_i^{-1}(s) \leq  G\circ \mathbf{F}_j^{-1}(t)\\
&\Rightarrow
G^{-1}\circ G\circ \mathbf{F}_i^{-1}(s) \leq  G^{-1}\circ  G\circ \mathbf{F}_j^{-1}(t)\\
&\Leftrightarrow
\mathbf{F}_i^{-1}(s) \leq   \mathbf{F}_j^{-1}(t)\\
&\Rightarrow
 G\circ \mathbf{F}_i^{-1}(s) \leq    G\circ \mathbf{F}_j^{-1}(t),
\end{align*}
where we used \reff{eq:d_i+-1} for the first equivalence, that $G^{-1}$ is non-decreasing for the first implication
and $G$ is non-decreasing for the second. Thus, we have that for $s,t\in
(0,1)$, that  the two conditions $\delta_{(i)}^{-1}(s)\leq
\delta_{(j)}^{-1}(t)$ and $\mathbf{F}_i^{-1}(s) \leq
\mathbf{F}_j^{-1}(t)$ are equivalent. Thus we deduce that the two sets
\[
\left\{(u_1, \ldots, u_d)\in I^d; \, \delta_{(1)}^{-1}(u_1)\leq  \dots \leq
\delta_{(d)}^{-1}(u_d)  \right\}
\]
and $T^{\mathbf{F}}$ are equal up to a set of zero Lebesgue measure.
Then we deduce \reff{eq:c_tilde_to_c} from  \reff{eq:kOS}.
\end{proof}

  We give a general result on the entropy of an exchangeable  random
  vector and the entropy of  its  order
  statistics.

   \begin{lem} \label{lem:HX_HXPI}
     Let $X$ be a random vector on $\R^d$, $X^{OS}$ the corresponding
     order statistics and $\Pi$ an independent uniform random variable
     on $\cs_d$. Then we have:
     \[
        H((X^{OS})_\Pi) =  \log(d!)+ H(X^{OS}) .
     \]
   \end{lem}

    \begin{proof}
      Let $F$ be  the cdf of $X^{OS}$.  If $F  \notin \cl_d^0$, then the
      cdf  $F^{sym}$ of  $(X^{OS})_\Pi$ given  by
      \reff{eq:Fsym} verifies  also $F^{sym} \notin  \cl_d^0$, therefore
      $H((X^{OS})_\Pi)=H(X^{OS})+ \log(d!)=-\infty$. If $F \in \cl_d^0$ with density
      function $f$, then the density  function $f^{sym}$ of $F^{sym}$ is
      given by, for $x\in \R^d$:
     \[
       f^{sym}(x) = \inv{d!} f(x^{OS}),
     \]
      where $x^{OS}$ is the ordered vector of $x$. Therefore, using that
      $f(x)=0$ if $x\neq x^{OS}$, we have:
      \begin{align*}
        H((X^{OS})_\Pi) & = -\int_{\R^d} f^{sym} \log(f^{sym}) \\
                 & = \log(d!) - \inv{d!} \int_{\R^d} f(x^{OS})
                 \log(f(x^{OS})) \, dx \\
& = \log(d!) -  \int_{\R^d} f(x)
                 \log(f(x) )\, dx \\
                 & = \log(d!) + H(X^{OS}).
      \end{align*}
    \end{proof}

    Now  we  are  ready  to  give  the  connection  between the entropy of
    $C$  and $S_{\mathbf{F}}(C)$  for   $C  \in  \cc^{OS}(\mathbf{F})$,
    which is the main result of this Section.   Recall  the
    definition  of $\delta^{\mathbf{F}}= (\delta^{\mathbf{F}}_{(i)},
    1\leq i\leq d)$ given in Lemma \ref{lem:delta_SFC} and thanks to
    Remark \ref{rem:dlogd}, $H(\delta^\mathbf{F}_{(i)})$ is finite for
    all $1\leq i\leq d$.

   \begin{prop} \label{prop:H(C)_H(SFC)}
      Let $\mathbf{F} \in \cf_d$ and  $C \in \cc^{OS}(\mathbf{F})$.
      Then we have:
   \begin{equation} \label{eq:h(K)}
        H(S_{\mathbf{F}}(C))  =    \log(d!)+ H(C)  + \sum_{i=1}^d H(\delta^{\mathbf{F}}_{(i)}).
   \end{equation}
   \end{prop}

  \begin{proof}
    Let  $U$  be  an exchangeable  random vector  with  cdf  $S_{\mathbf{F}}(C)$,  and
    $U^{OS}$    its    order     statistics.    According    to    Lemma
    \ref{lem:delta_SFC}, $U^{OS}$  has one-dimensional marginals  cdf
    $\delta^{\mathbf{F}}=(\delta^{\mathbf{F}}_{(i)}, 1\leq i\leq d)$ and
    copula $C$. Therefore, using \reff{eq:entr_decomp}, we get:
     \begin{equation*}
       H(U^{OS}) =  H(C) + \sum_{i=1}^d H(\delta^{\mathbf{F}}_{(i)}).
     \end{equation*}
     On the  other hand,  since $S_{\mathbf{F}}(C)$ is  symmetric, Lemma
     \ref{lem:exch_OS}   ensures  that   $(U^{OS})_\Pi$  has   the  same
     distribution as $U$. Therefore Lemma \ref{lem:HX_HXPI} gives:
      \begin{equation*}
       H(S_{\mathbf{F}}(C)) = H(U) =H\left((U^{OS})_\Pi\right) =  H(U^{OS}) +\log(d!) =  H(C) + \sum_{i=1}^d H(\delta^{\mathbf{F}}_{(i)}) +\log(d!).
      \end{equation*}

  \end{proof}

 \section{Maximum entropy copula with given multidiagonals} \label{sec:max_ent_copula}

 This section is a generalization of \cite{butucea2013maximum}, where the
 maximum entropy copula with given diagonal section (i.e. given distribution
 for the maximum of its marginals) is studied.

 Recall  that multidiagonals  of copulas  on  $\R^ d$  are given by
 Definition  \ref{defi:diag}.  We recall some further
 notation: $\cd$ denotes the set of multidiagonals; for $\delta\in \cd$,
 $\cc_\delta$ denotes the subset of copulas with multidiagonal $\delta$;
 $\cc^0$ denotes the subset copulas  which are absolutely continuous,
 and $\cc^0_\delta = \cc_\delta \cap \cc^0$. The set $\cd^0 \subset \cd$
 contains all diagonals for which $\cc_\delta^0 \neq \emptyset$.

 We  give an explicit formula for  $C^*$   such   that
 $H(C^*)=\max_{C\in  \cc_\delta}  H(C)$,  with   $H$  the  entropy,  see
 definition \reff{eq:def-entropy}. Notice that the  maximum can be taken over
 $\cc_\delta^0$, since the entropy is minus infinity otherwise.
  When $d=2$, the problem was solved in \cite{butucea2013maximum}.

  Let $\delta=(\delta_{(i)}, 1\leq i\leq  d)\in \cd$ be a multidiagonal.
  Since $\delta_{(i)}$,  $1\leq i\leq d$ are  $d$-Lipschitz, the entropy
  of  $H  (\delta_{(i)})$  is  well   defined  and  finite,  see  Remark
  \ref{rem:dlogd} and $\cj(\delta)$ given  by \reff{eq:cj_delta} is also
  well defined and  belongs to $[0,+\infty]$.

 The next two lemmas provides sets on which the density of a copula with
 given multidiagonal is zero. For $\delta\in \cd$, let:
\begin{equation}
   \label{eq:def-Z}
Z_\delta=\{u \in I^d;\,  \text{ there exists $1 \leq i \leq d$ such that } \delta_{(i)}'(u_{(i)})=0\}.
\end{equation}

\begin{lem}
   \label{lem:0}
Let $\delta \in  \cd^0$. Then  for all copulas $C \in \cc^0_\delta$  with density $c$,
  we  have $c\ind_{Z_\delta}=0$ a.e. that is $c(u)\ind_{Z_\delta}(u)=0$
  for a.e. $u\in I^d$.
\end{lem}
 \begin{proof}
By definition of $\delta_{(i)}$, we have for all
$r\in I$:
\[
\int_{I^d} c(u) \ind_{\{u_{(i)} \leq r\}} \, du = \delta_{(i)}(r)=\int_0^r
\delta_{(i)}'(s)\, ds.
\]
This implies, by the monotone class theorem,  that for all
measurable subset $K$ of $I$, we have:
\[
\int_{I^d} c(u) \ind_K(u_{(i)}) \, du = \int_K
\delta_{(i)}'(s)\, ds.
\]
Since $c\geq 0$ a.e., we deduce that
a.e. $c(u)\ind_{\{\delta_{(i)}'(u_{(i)})=0\}}=0$ and thus  a.e. $c\ind_{Z_\delta} = 0 $.
\end{proof}

Recall  the definition of
  $\Psi^\delta_i$ given by \reff{eq:def_Psi_i_F} for  $2 \leq i \leq d$.
  We  also define  $\Psi^\delta_1 =  (0,d_1)$ with  $d_1=\inf\{s \in  I;
  \delta_{(1)}(s)  =1\}$  and  $\Psi^\delta_{d+1}  =  (g_{d+1},1)$  with
  $g_{d+1}=\sup\{s \in I; \delta_{(d)}(s)  =0\}$ . Since $\Psi^\delta_i$
  are open subsets of $I$, there  exists at most countably many disjoint
  intervals $\{ (g_i^{(j)}, d_i^{(j)})$, $j \in J_i \}$ such that
\begin{equation} \label{eq:def_alpha_beta}
   \Psi^\delta_i= \bigcup_{j \in J_i} (g_i^{(j)}, d_i^{(j)}).
\end{equation}
We  note by  $m_i^{(j)}=(g_i^{(j)}+d_i^{(j)})/2$ the  midpoint of  these
intervals    for    $2    \leq    i    \leq    d+1$.    In    particular
$m_{d+1}=(1+g_{d+1})/2$. We also define  $m_1=0$.
For $\delta \in \cd$, let:
\begin{equation} \label{eq:def_LD}
   L_\delta=\{u=(u_1,\hdots, u_d) \in I^d; (u_{(i-1)},u_{(i)}) \subset  \Psi^\delta_i  \text{ for all } 2 \leq i \leq d\}.
\end{equation}
We have the following
Lemma for  all absolutely  continuous copula  $C \in  \cc^0_\delta$ with
density $c$.

\begin{lem} \label{lem:c_zero_Psi} Let $\delta \in  \cd^0$ and $2 \leq i
  \leq d$. Then  for all copulas $C \in \cc^0_\delta$  with density $c$,
  we  have $c\ind_{I\setminus L_\delta}=0$ a.e., that is for a.e.
  $u=(u_1, \hdots,  u_d) \in  I^d$, for  all $s
  \notin \Psi^\delta_i$ :
\begin{equation} \label{eq:c_zero_Psi}
  c(u)\ind_{\{u_{(i-1)} < s < u_{(i)}\}} = 0.
\end{equation}
\end{lem}

\begin{proof}
 The complementary set $(\Psi_i^\delta)^c$ is  given by:
\begin{equation} \label{eq:Psi^c}
  (\Psi_i^\delta)^c = \overline{\bigcup_{j \in J_i} \{g_i^{(j)}, d_i^{(j)}\}}.
\end{equation}
Let $U=(U_1, \hdots, U_d)$ be a random vector with cdf $C \in \cc^0_\delta$.
For  $2 \leq i \leq d$ and $s \in \bigcup_{j \in J_i} \{g_i^{(j)}, d_i^{(j)}\}$, that is $\delta_{(i-1)}(s)=\delta_{(i)}(s)$, we have:
 \[
    \P(U_{(i-1)} < s < U_{(i)}) = \P(U_{(i-1)} < s)- \P(U_{(i)} \leq s) = \delta_{(i-1)}(s) - \delta_{(i)}(s) = 0.
 \]
This implies that \reff{eq:c_zero_Psi} holds a.e. for all $s \in \bigcup_{j \in J_i} \{g_i^{(j)}, d_i^{(j)}\}$. Since $J_i$ is at most countable,
we have for a.e. $u \in I^d$ and for all $s \in \bigcup_{j \in J_i} \{g_i^{(j)}, d_i^{(j)}\}$, that
\reff{eq:c_zero_Psi} holds. Since for all $u \in I$, $s \notin \Psi^\delta_i$ there exists $s' \in \bigcup_{j \in J_i} \{g_i^{(j)}, d_i^{(j)}\}$ such that
\[
\ind_{\{u_{(i-1)} < s < u_{(i)}\}}=\ind_{\{u_{(i-1)} < s' < u_{(i)}\}},
\]
we can conclude that for a.e.  $u \in I^d$ and for all $s \notin \Psi^\delta_i$ \reff{eq:c_zero_Psi} hold.

\end{proof}

Notice that for all $u=(u_1,\hdots, u_d) \in I^d$:
\begin{equation}\label{eq:L_delta_leq_prod}
  \ind_{L_\delta}(u) \leq \prod_{i=1}^d \ind_{\Psi^\delta_i \cap \Psi^\delta_{i+1}}(u_{(i)}).
\end{equation}
We define the function
$c_\delta$ on $I^d$ as, for $u=(u_1, \hdots, u_d)\in I^d$:
\begin{equation} \label{eq:c_delta}
  c_\delta(u)=\inv{d!} \ind_{L_\delta}(u)\prod_{i=1}^d a_i(u_{(i)}),
\end{equation}
where the function $a_i$, $1 \leq i \leq d$, are given by, for $t \in I$:
\begin{equation}
\label{eq:ai}
  a_i(t) =  K'_i(t) \expp{K_{i+1}(t)-K_i(t)} \ind_{ \Psi_i^\delta \cap \Psi_{i+1}^\delta}(t),
\end{equation}
  with for $1 \leq i \leq d$, $t \in (g_i^{(j)},d_i^{(j)})$:
 \begin{equation} \label{eq:F}
   K_i(t)=\int_{m_i^{(j)}}^t \frac{\delta_{(i)}'(s)}{\delta_{(i-1)}(s)-\delta_{(i)}(s)} \,ds
 \end{equation}
and the conventions $\delta_{(0)}=1$ and $K_{d+1}=0$. Notice that for
$t\in \Psi^\delta_1$:
\begin{equation}
   \label{eq:K1}
K_1(t)=-\log(1-\delta_{(1)}(t)).
 \end{equation}
\begin{rem}
  The choice  of ${m_i^{(j)}}$ for
  the   integration   lower bound   in    \reff{eq:F}   is    arbitrary
  as  any  other value in
  $(g_i^{(j)},d_i^{(j)})$
  would   not change  the definition of  $c_\delta$ in
  \reff{eq:c_delta}.
\end{rem}

\begin{rem} \label{rem:d^i_t-}
  For all $1 \leq i \leq d$, $j \in J_i$, $t \in (m_i^{(j)},d_i^{(j)})$, we have the following lower bound for $K_i(t)$:
 \begin{equation*}
     K_i(t) \geq \int_{m_i^{(j)}}^t \frac{\delta_{(i)}'(s)}{\delta_{(i-1)}(d_i^{(j)})-\delta_{(i)}(s)} \,ds = \log\left(\frac{\delta_{(i-1)}(d_i^{(j)})-\delta_{(i)}(m_i^{(j)})}{\delta_{(i-1)}(d_i^{(j)})-\delta_{(i)}(t)}\right).
 \end{equation*}
 Since $\delta_{(i)}$ is non-decreasing and $\delta_{(i-1)}(d_i^{(j)}) = \delta_{(i)}(d_i^{(j)})$, we have $\lim_{t \nearrow d_i^{(j)}} K_i(t) = + \infty$.
\end{rem}

The following Proposition states that $c_\delta$ is the density of an
absolutely continuous symmetric copula $C_\delta \in \cc_\delta^0 \cap
\cc^{sym}$. It is more general than the results in  \cite{butucea2013maximum}, where
only the diagonal $\delta_{(d)}$ was supposed given.

\begin{prop} \label{prop:c_delta}
 Let $\delta \in \cd^0$. The function $c_\delta$ defined in \reff{eq:c_delta}-\reff{eq:F}
   is the density of a symmetric copula $C_\delta \in \cc_\delta^0 \cap \cc^{sym}$. In addition, we have:
   \begin{equation} \label{eq:H(c)}
     H(C_\delta) = -\cj(\delta) + \log(d!) + (d-1)+ \sum_{i=1}^d H
     (\delta_{(i)}) .
   \end{equation}
\end{prop}

The proof  of this Proposition is given in Section \ref{sec:appendixA}.
The following characterization of $C_\delta$ is proved in Section
\ref{sec:proof-prod}.

\begin{prop} \label{prop:c_delta_unique_product}
Let $\delta \in \cd^0$. Then
$C_\delta$  is the only copula in $\cc_\delta^0$ whose density is of the
form $(1/d!)\ind_{L_\delta}(u)\prod_{i=1}^d h_i(u_{(i)})$, where $h_i$, $1 \leq i \leq d$ are
measurable non-negative functions defined on $I$.
\end{prop}

The following Theorem states that
the unique optimal solution of $\max_{C\in \cc_\delta}(H(C))$, if it
exists,  is  given  by  $C_\delta$.   Its  proof  is  given  in  Sections
 \ref{sec:case-a} for case (a) and \ref{sec:proof-spec} for case (b).

\begin{theo} \label{theo:spec}
 Let $\delta \in \cd$.
 \begin{itemize}
   \item[(a)] If $ \cj(\delta)= +\infty$ then
     $\max_{C \in \cc_\delta} H(C)=-\infty $.
   \item[(b)] If $\cj(\delta) < +\infty$ then $\delta \in \cd^0$,
$ \max_{C \in \cc_\delta} H(C) > -\infty $ and  $C_\delta$ given in Proposition \ref{prop:c_delta} is the unique copula such that
$H\left(C_\delta\right)=\max_{C \in \cc_\delta} H(C)$.
\end{itemize}
\end{theo}

The copula $C_\delta$ will be called the maximum entropy copula with
given multidiagonal.

\section{Maximum entropy distribution of order statistics with given marginals} \label{sec:max_entr_dens_ord_stat}

We use the results of Section \ref{sec:max_ent_copula}
to compute the density of the maximum entropy copula for
marginals $\mathbf{F} \in \cf_d^0$ with $\cf_d^0$ defined in Section \ref{sec:dens+entropy}. Recall $\delta^\mathbf{F}=(\delta_{(1)},\hdots,\delta_{(d)})=\mathbf{F}\circ
G^{-1}$ and the definition of  $\Sigma^{\delta^\mathbf{F}}$
in \reff{eq:def_sigma_F}.  Recall $K_i$ defined by \reff{eq:F}, for $1\leq i\leq d$ and $T^{\mathbf{F}}$ defined by \reff{eq:c_zero-T}.
% by\reff{eq:c_zero-T}.
We define the function
$c_\mathbf{F}$ on $I^d$, for $u=(u_1, \ldots, u_d)\in I^d$:
\begin{equation}
  \label{eq:def-cF}
   c_\mathbf{F}(u) =\prod_{i=2}^d
   \frac{\expp{K_i(\delta_{(i-1)}^{-1}(u_{i-1}))-K_i(\delta_{(i)}^{-1}(u_{i}))
     }}{\delta_{(i-1)}\circ \delta_{(i)} ^{-1}(u_i) -u_i}
            \ind_{\{u \in T^\mathbf{F}; ( \delta_{(1)}^{-1}(u_1),
             \hdots,\delta_{(d)}^{-1}(u_d)) \in L_{\delta^\mathbf{F}}\}}\,
     \ind_{\{\prod_{i=1}^d \delta_{(i)}'\circ \delta_{(i)}^{-1}(u_i)> 0\}}.
\end{equation}

  Recall the function $\cj(\delta)$ defined on the set of multidiagonals
  by $\reff{eq:cj_delta}$ and $C_{\delta^{\mathbf{F}}}$ the copula with
  density given by \reff{eq:c_delta}-\reff{eq:F}.
\begin{prop}
   \label{prop:c_F-density}
 Let $\mathbf{F}\in \cf_d^0$ .
    The  function $c_\mathbf{F}$ defined by \reff{eq:def-cF} is the
    density of the copula
    $C_\mathbf{F}=\cs^{-1}_{\mathbf{F}}(C_{\delta^{\mathbf{F}}}) $ which
    belongs to
    $\cc^{OS}(\mathbf{F})$. The entropy of $C_\mathbf{F}$ is given by:
 \begin{equation}
   \label{eq:H-CF}
     H(C_\mathbf{F})= d-1 - \cj(\delta^{\mathbf{F}}).
\end{equation}
\end{prop}

\begin{proof}
  Since  $\mathbf{F} \in  \cf_d^0$,  we have that
  $\delta^{\mathbf{F}} \in \cd^0$. According     to
  Proposition~\ref{prop:c_delta},      $c_{\delta^{\mathbf{F}}}$      defined      by
  \reff{eq:c_delta}   is    the   density   of   a    symmetric   copula
  $C_{\delta^{\mathbf{F}}}$  which  belongs to  $\cc^{sym}(\mathbf{F}) \cap \cc^0$,
  thanks  to  Proposition   \ref{prop:char_cfsym} and Lemma \ref{lem:dens_c_ctilde}.  According  to  Lemma
  \ref{lem:dens_c_ctilde},  formula \reff{eq:c_tilde_to_c}  we get  that
  $c_{\mathbf{F}}=s^{-1}_{\mathbf{F}}(C_{\delta^{\mathbf{F}}})$     is
  therefore the  density of a  copula $C_{\mathbf{F}}$ which  belongs to
  $\cc^{OS}(\mathbf{F}) \cap \cc^0$.   Use  \reff{eq:c_tilde_to_c} and  \reff{eq:cont_J}  to
  check \reff{eq:def-cF}.
To conclude, use \reff{eq:h(K)} and \reff{eq:H(c)} to get \reff{eq:H-CF}.
   \end{proof}

   Analogously to Lemma \ref{lem:c_zero_Psi}, we have the following
   restriction on the support of all $F \in \cl_d^{OS}(\mathbf{F})
   \bigcap \cl^0_d$. Recall the definition of $\Psi^\mathbf{F}_i$ in
   \reff{eq:def_Psi_i_F}. The proof of the next Lemma is similar to the
   proof of Lemma \ref{lem:c_zero_Psi} and is left to the reader.
\begin{lem}
  Let $\mathbf{F} \in \cf_d^0$ and $t \in \mathbf{F}_{i}\left((\Psi^\mathbf{F}_i)^c\right)$ for some $2 \leq i \leq d$. Then we have for all $F \in \cl_d^{OS}(\mathbf{F}) \bigcap \cl^0_d$ with density function $f$:
  \[
     f(x)\ind_{\{ x_{i-1} < t < x_{i}\}}=0 \quad \text{ for a.e. } x=(x_1,\hdots, x_d) \in S.
  \]
\end{lem}
For $\delta \in \cd$, recall the definition of $L_{\delta}$ in \reff{eq:def_LD}. Let $L^{\delta}=L_{\delta} \cap S$. More generally, for $\mathbf{F} \in \cf_d$, we set:
\begin{equation} \label{eq:def_LF}
   L^\mathbf{F}=\{x=(x_1,\hdots, x_d) \in S; (x_{i-1},x_{i}) \subset \Psi^\mathbf{F}_i  \text{ for all } 2 \leq i \leq d\}.
\end{equation}
The next  Lemma establishes the connection between the sets $L_{\delta^\mathbf{F}}$ defined by \reff{eq:def_LD} and $L^\mathbf{F}$.

\begin{lem} \label{lem:LD_LF}
 Let $\mathbf{F}=(\mathbf{F}_1, \hdots, \mathbf{F}_d) \in \cf_d^0$ with density functions $\mathbf{f}_i$ for $1 \leq i \leq d$. Let  $\delta^\mathbf{F}=(\delta_{(1)}, \hdots, \delta_{(d)})$ given by \reff{eq:d_i+-1}, $T^\mathbf{F}$ given by \reff{eq:c_zero-T} and $L_{\delta^\mathbf{F}}$ defined by \reff{eq:def_LD}. Then for
  $ \prod_{i=1}^d \mathbf{f}_i(x_i)dx_1 \hdots x_d$-a.e. $x \in \R^d$ we have that
 \[
    \ind_{T^\mathbf{F}}\left(\mathbf{F}_1(x_1), \hdots, \mathbf{F}_d(x_d)\right) \ind_{L_{\delta^\mathbf{F}}}\left(\delta_{(1)}^{-1} \circ \mathbf{F}_1(x_1), \hdots,\delta_{(d)}^{-1} \circ \mathbf{F}_d(x_d)\right) =
    \ind_{L^\mathbf{F}}(x)
 \]

\end{lem}

\begin{proof}

According to \reff{eq:IgJc} and  \reff{eq:inv2}, we have $\mathbf{f}_{i} (t)\, dt$-a.e. that
$\mathbf{F}_{i}^{-1}
\circ\mathbf{F}_{i}(t)=t$.
This implies  that $\prod_{i=1} \mathbf{f}_{i} (x_i)\, dx_1\cdots
dx_d$-a.e., $(\mathbf{F}_{1}(x_1), \ldots, \mathbf{F}_{d}(x_d))$
belongs to $T^{\mathbf{F}}$  if and only if $x \in S$.
Recall the sets $\Psi_i^{\delta^\mathbf{F}}$ given by \reff{eq:def_Psi_i_F}. For $\prod_{i=1} \mathbf{f}_{i} (x_i)\, dx_1\cdots
dx_d$-a.e. $x \in S$, we have:
\begin{align*}
& (\delta_{(1)}^{-1}\circ\mathbf{F}_1(x_1),
\hdots,\delta_{(d)}^{-1}\circ\mathbf{F}_d(x_d)) \in
L_{\delta^\mathbf{F}}   \\
\Longleftrightarrow & \left(\delta_{(i-1)}^{-1}\circ\mathbf{F}_{i-1}(x_{i-1}),
  \delta_{(i)}^{-1}\circ\mathbf{F}_i(x_i)\right)
\subset \Psi_i^{\delta^\mathbf{F}}, \quad 2 \leq i \leq  d\\
\Longleftrightarrow  & \forall t\in
 \left(\delta_{(i-1)}^{-1}\circ\mathbf{F}_{i-1}(x_{i-1}),\delta_{(i)}^{-1}
   \circ \mathbf{F}_i(x_i)\right) : \delta_{(i-1)}(t) > \delta_{(i)}(t),
 \quad 2 \leq i \leq  d\\
\Longleftrightarrow & \forall t\in \left(G \circ
  \mathbf{F}^{-1}_{i-1}\circ\mathbf{F}_{i-1}(x_{i-1}),G \circ
  \mathbf{F}^{-1}_i \circ \mathbf{F}_i(x_i)\right) : \mathbf{F}_{i-1}
\circ G^{-1} (t) > \mathbf{F}_{i} \circ G^{-1} (t), \quad  2 \leq i
\leq  d  \\
\Longleftrightarrow & \forall t \in \left(G(x_{i-1}),G(x_i)\right) :
\mathbf{F}_{i-1} \circ  G^{-1} (t) > \mathbf{F}_{i} \circ G^{-1} (t),
\quad  2 \leq i \leq  d,
\end{align*}
where   the   first   equivalence   comes   from   the   definition   of
$L_{\delta^\mathbf{F}}   $,   the   second  from   the   definition   of
$\psi_i^{\delta^\mathbf{F}}$,  the third  from \reff{eq:d_i+-1}  and the
last      from     the      fact     that      $\mathbf{f}_{i}     (t)\,
dt$-a.e.  $\mathbf{F}_{i}^{-1} \circ\mathbf{F}_{i}(t)=t$.   Consider the
change of variable $s=G^{-1}(t)$. We have by \reff{eq:inv}:
\[
  t < G(x_i) \Longleftrightarrow G^{-1}(t) < x_i \Longleftrightarrow s <x_i.
\]
Since $x_{i-1} \in I_g(\mathbf{F}_{i-1})$  $\mathbf{f}_{i-1}     (x_{i-1})\,
dx_{i-1}$-a.e., we get  $x_{i-1} \in I_g(G)$ and by \reff{eq:inv2}:
\[
  G(x_{i-1}) < t \Longleftrightarrow x_{i-1} < G^{-1}(t) \Longleftrightarrow x_{i-1} < s .
\]
Therefore we deduce that  $\prod_{i=1} \mathbf{f}_{i} (x_i)\, dx_1\cdots
dx_d$-a.e. $x \in S$:
\begin{align*}
&(\delta_{(1)}^{-1}\circ\mathbf{F}_1(x_1),
\hdots,\delta_{(d)}^{-1}\circ\mathbf{F}_d(x_d)) \in
L_{\delta^\mathbf{F}}    \\
\Longleftrightarrow &  \forall s \in \left(x_{i-1},x_i\right) :
\mathbf{F}_{i-1} (s) > \mathbf{F}_{i} (s), \quad 2 \leq i \leq  d\\
 \Longleftrightarrow & x \in L^\mathbf{F}.
\end{align*}
\end{proof}

Using Proposition \ref{prop:H(C)_H(SFC)}, we check that the copula $C_{\mathbf{F}}$ maximizes the entropy over the
set $\cc^{OS}(\mathbf{F})$. We set for $x=(x_1, \hdots, x_d)\in \R^d$:
\begin{equation}
   \label{eq:def-FF}
            F_{\mathbf{F}}(x)=C_{\mathbf{F}}({\mathbf{F}}_{1}(x_1), \hdots,{\mathbf{F}}_{d}(x_d)).
\end{equation}

Let $\mathbf{f}_i$ denote the density function of $\mathbf{F}_i$ when it
exists. Let us further note for $2 \leq i \leq d$, $t \in \R$:
\begin{equation} \label{eq:def_l}
  \ell_i(t) = \frac{\mathbf{f}_i(t)}{\mathbf{F}_{i-1}(t) - \mathbf{F}_i(t)}\cdot
\end{equation}
When the densities $\mathbf{f}_i$ exist for all $1 \leq i \leq d$, we define the function $f_{\mathbf{F}}$ for $x=(x_1,\hdots, x_d) \in \R^d$ as:
   \begin{equation} \label{eq:def_fF}
     f_{\mathbf{F}}(x)=
            \mathbf{f}_{1} (x_1) \prod_{i=2}^d \ell_i(x_i)  \exp\left(-\int_{x_{i-1}}^{x_i}
           \ell_i(s)  \, ds \right) \ind_{L^\mathbf{F}}(x),
   \end{equation}
   with  $L^\mathbf{F}$  given  by \reff{eq:def_LF}.  The  next  theorem
   asserts that the cdf $F_{\mathbf{F}}$  maximizes the entropy over the
   set   $\cl_{d}^{OS}(\mathbf{F})$   and   that  its   density   is   $
   f_{\mathbf{F}}$.  Recall $\cj$ defined by \reff{eq:cj_delta}.

\begin{theo}
\label{theo:f*}
Let  $\mathbf{F}=(\mathbf{F}_i,  1\leq     i\leq d) \in \cf_d$.
     \begin{enumerate}
        \item[(a)] If there exists $1\leq i \leq d$ such that
          $H(\mathbf{F}_i)=-\infty$, or if $\cj({\mathbf{F}}) = +
          \infty$, then we have $\max _{F \in \cl_{d}^{OS}(\mathbf{F})}
          H(F)=-\infty$.
        \item[(b)] If  $H(\mathbf{F}_i) >  - \infty$ for  all $1  \leq i
           \leq  d$, and  $\cj({\mathbf{F}}) <  + \infty$,  then we  have
           $\mathbf{F}       \in       \cf_d^0$,       $\max_{F       \in
             \cl_{d}^{OS}(\mathbf{F})}    H(F)>     -\infty$,    and    $
           F_{\mathbf{F}}$ defined in \reff{eq:def-FF}  is the unique cdf
           in $  \cl_{d}^{OS}(\mathbf{F})  $ such             that
         $H(F_{\mathbf{F}})=\max_{F\in
             \cl_{d}^{OS}(\mathbf{F})}  H(F)$.  Furthermore,  the density
           function  of   $F_{\mathbf{F}}$  exists,   and  is   given  by
           $f_{\mathbf{F}}$ defined in \reff{eq:def_fF}. We also have:
\[
H(F_{\mathbf{F}})=d-1  + \sum_{i=1}^d H({\mathbf{F}_i}) -
\cj(\mathbf{F}).
\]
     \end{enumerate}
   \end{theo}
\begin{proof}
The proof of case (a) is postponed to Section \ref{sec:case-a}.

We shall assume that $H(\mathbf{F}_i) > - \infty$ for all $1 \leq i \leq
d$ and  $\cj(\delta^{\mathbf{F}}) <  + \infty$.   This implies  that the
densities  $\mathbf{f}_i$ of  $\mathbf{F}_i$ exist  for $1\leq  i\leq d$
and, thanks to Lemma \ref{lem:F^0},  that $\mathbf{F} \in \cf_d^0$.  Let
$F_{\mathbf{F}}$ be  defined by \reff{eq:def-FF},  that is the  cdf with
copula  $C_{\mathbf{F}}$  from  Proposition  \ref{prop:c_F-density}  and
one-dimensional  marginals  cdf  ${\mathbf{F}}$. Thanks  to  Proposition
\ref{prop:c_F-density},        we         have        $F_{\mathbf{F}}\in
\cl_{d}^{OS}(\mathbf{F})$.

  We      deduce      from      \reff{eq:entr_decomp},      Propositions
  \ref{prop:char_cfsym}    and   \ref{prop:H(C)_H(SFC)},   Theorem
  \ref{theo:spec} case (b) and Proposition
  \ref{prop:c_F-density} that  $F_{\mathbf{F}}$  is  the only  cdf  such  that
  $H(F_{\mathbf{F}})=\max_{F\in \cl_{d}^{OS}(\mathbf{F})} H(F)$.
We deduce  from \reff{eq:entr_decomp},
     \reff{eq:H-CF} and Lemma \ref{lem:J(F)} that:
\[
H(F_{\mathbf{F}})=d-1  + \sum_{i=1}^ d H({\mathbf{F}_i}) -
\cj(\mathbf{F}).
\]

Since the copula $C_{\mathbf{F}}$  is absolutely continuous with density
$c_{\mathbf{F}}$   given    in   \reff{eq:def-cF},   we    deduce   from
\reff{eq:def-FF} that $F_{\mathbf{F}}$ has density $f_\mathbf{F}$ given
by, for a.e. $x=(x_1, \ldots,, x_d) \in \R^d$:
\begin{equation} \label{eq:f*_c*_fi}
       f_\mathbf{F}(x)= c_\mathbf{F}\left( \mathbf{F}_{1}(x_1),
         \hdots,\mathbf{F}_{d}(x_d)\right)
        \prod_{i=1}^d \mathbf{f}_{i} (x_i).
\end{equation}
Recall the expression \reff{eq:def-cF} of $c_\mathbf{F}$ as well as
$K_i$ defined  by \reff{eq:F}, for $1\leq i\leq d$.
Using the change of variable $s=G^{-1}(t)$ and \reff{eq:FiG-1G=Fi}, we
get (similarly to the proof of Lemma \ref{lem:J(F)}):
\begin{equation}
   \label{eq:K}
K_i\circ
\delta_{(i)}^{-1} \circ \mathbf{F}_{i}(x_i) - K_i\circ
\delta_{(i-1)}^{-1} \circ \mathbf{F}_{i-1}(x_{i-1})
= \int_{\mathbf{F}_{i-1}^{-1}\circ \mathbf{F}_{i-1}(x_{i-1})  }
^{\mathbf{F}_{i}^{-1}\circ \mathbf{F}_{i}(x_{i})  }
\ell_i(s)\,ds.
\end{equation}
Using \reff{eq:FiG-1G=Fi}, we also get:
\begin{equation}
   \label{eq:df}
\frac{\mathbf{f}_{i} (x_i)}{\delta_{(i-1)} \circ \delta_{(i)}^{-1} \circ\mathbf{F}_{i}
  (x_i)
  -\mathbf{F}_{i}(x_i)}
= \frac{\mathbf{f}_{i} (x_i)}{\mathbf{F}_{i-1}\circ \mathbf{F}_{i}^{-1} \circ\mathbf{F}_{i}
  (x_i)
  -\mathbf{F}_{i}(x_i)} \cdot
\end{equation}
According to \reff{eq:inv2}, we have $\mathbf{f}_{i} (t)\, dt$-a.e. that
$\mathbf{F}_{i}^{-1}
\circ\mathbf{F}_{i}(t)=t$.
For $1 \leq i \leq d$, we have from \reff{eq:inv2} that  $\prod_{i=1} \mathbf{f}_{i} (x_i)\, dx_1\cdots
dx_d$-a.e.:

\begin{equation} \label{eq:ind_delta'=1}
   \ind_{\left\{\prod_{i=1}^d \delta_{(i)}'\circ \delta_{(i)}^{-1}(\mathbf{F}_{i}(x_i))> 0\right\}} = 1.
\end{equation}
We deduce from \reff{eq:def-cF}, \reff{eq:f*_c*_fi}, \reff{eq:K}, \reff{eq:df}, \reff{eq:ind_delta'=1} and Lemma \ref{lem:LD_LF} that a.e. for $x=(x_1, \ldots, x_d)\in
\R^d$:
\[
f_\mathbf{F}(x)
=
\mathbf{f}_{1} (x_1) \prod_{i=2}^d
\ell_i(x_i)
 \expp{-\int_{x_{i-1}}^{x_i}
\ell_i(s) \, ds } \, \ind_{L^\mathbf{F}}(x).
\]
\end{proof}

\begin{rem}
   \label{rem:f_F=density}
We deduce from the proof of Theorem \ref{theo:f*}  case (b) and
Proposition \ref{prop:c_F-density}, that if
$\mathbf{F}=(\mathbf{F}_i,  1\leq     i\leq d) \in \cf_d^0$, then
$f_\mathbf{F}$ defined by \reff{eq:def_fF} is a  probability density function on
$S\subset \R^d$.
\end{rem}

\begin{rem} \label{rem:f*=prod_pi}
     The density $f_\mathbf{F}$ has a product form on $L^\mathbf{F}$, that is it
     can be written as, for
     a.e. $x=(x_1,\hdots,x_d) \in \R^d$:
\begin{equation}\label{eq:f*=prod_pi}
f_\mathbf{F}(x) = \prod_{i=1}^d p_i(x_i) \ind_{L^\mathbf{F}}(x) ,
\end{equation}
     where the functions $(p_i, 1\leq i\leq d)$ are measurable and
     non-negative.
\end{rem}

  In addition to Remark \ref{rem:f*=prod_pi}, the next
  Corollary asserts that $F_{\mathbf{F}}$ is the only element of
  $\cl_{d}^{OS}(\mathbf{F})$, whose density has a product form.

\begin{cor}
\label{cor:f*_uni_prod}
  Let $\mathbf{F} \in \cf_d^0$. Let $F \in \cl_{d}^{OS}(\mathbf{F})$ be an absolutely continuous cdf with density $f$ given by, a.e. for $x=(x_1,\hdots, x_d) \in \R^d$:
  $f(x)=\prod_{i=1}^d h_i(x_i)\ind_{L^\mathbf{F}}(x)$,
  with $h_i$, $1 \leq i \leq d$ some measurable non-negative functions
  on $\R$. Then we have $F=F_{\mathbf{F}}$  on $\R^d$.
  \end{cor}

  \begin{proof}
   Let $X=(X_1, \ldots, X_d)$ be an order statistic with cdf $F$, and $F^{sym}$ the cdf of
   $X_\Pi$ given by \reff{eq:Fsym}, with $\Pi$ uniform on $\cs_d$ and
   independent of $X$.
   Then the cdf $F^{sym}$ is also absolutely continuous, and its density $f^{sym}$ is given by :
   \begin{equation} \label{eq:tilde_f=prod_qi}
     f^{sym}(x)=\inv{d!} \prod_{i=1}^d h_i(x_{(i)})\ind_{L^\mathbf{F}}(x^{OS}),
   \end{equation}
where $x^{OS}$ is the ordered vector of $x$. The one-dimensional marginal cdf's of $X_\Pi$ are all equal to $G$ given by \reff{eq:def_G}.
Let $C\in\cc^{OS}(\mathbf{F}) \cap \cc^0$ denote the copula of $F$. Then according to \reff{eq:inversion_formula}, the copula $S_{\mathbf{F}}(C)$ of $X_\Pi$ is given by, for a.e. $u=(u_1, \ldots, u_d)\in I^d$:
\[
   S_{\mathbf{F}}(C)(u)=F^{sym}\left(G^{-1}(u)\right).
\]
Therefore its  density $s_{\mathbf{F}}(C)$ can be expressed as:
\begin{align*}
s_{\mathbf{F}}(C)(u) & = \frac{ f^{sym}(G^{-1}(u))}{\prod_{i=1}^d g\circ G^{-1}(u_i)} \prod_{i=1}^d \ind_{\{g \circ G^{-1}(u_{i})>0\}} \\
                                    & = \inv{d!} \ind_{L^\mathbf{F}}(G^{-1}(u^{OS})) \prod_{i=1}^d\frac{h_i\circ G^{-1}(u_{(i)})}{g\circ G^{-1}(u_{(i)})}\ind_{\{g \circ G^{-1}(u_{(i)})>0\}},
\end{align*}
where $g$ is the density of $G$. Notice that for $x=(x_1,\hdots,x_d) \in S$, we have a.e.:
\[
   \ind_{\{\prod_{i=1}^d h_i(x_i)>0\}} \leq \ind_{\{\prod_{i=1}^d f_i(x_i)>0\}} ,
\]
with $f_i$ the density of  $X_i$.
Therefore by Lemma \ref{lem:LD_LF} and since $G$ is continuous, we
have  that  $\prod_{i=1}^d  h_i \circ  G^{-1}(u_{(i)})  \,du_{1}\hdots
du_{d}$-a.e.:
\[
  \ind_{L^\mathbf{F}}(G^{-1}(u^{OS})) = \ind_{L_{\delta^\mathbf{F}}}(
\delta_ {(1)}^{-1}\circ \delta_{(1)}(u_{(1)}), \ldots,
\delta_ {(d)}^{-1}\circ \delta_{(d)}(u_{(d)})).
\]
By  Lemma
\ref{lem:delta_SFC}, $S_{\mathbf{F}}(C)$ belongs to
$\cc^0_{\delta^{\mathbf{F}}}$ and thus $s_{\mathbf{F}}(C)=0$ a.e. on
$Z_{\delta^\mathbf{F}}$ defined by \reff{eq:def-Z}. Then use
\reff{eq:IgJc} and \reff{eq:inv2} to get that $\delta_ {(i)}^{-1}\circ
\delta_{(i)}(u_{(i)})=u_{(i)}$ a.e. on $Z_{\delta^\mathbf{F}}^c$. This
gives:
\[
  \ind_{L^\mathbf{F}}(G^{-1}(u^{OS})) =
  \ind_{L_{\delta^\mathbf{F}}}(u^{OS})=
  \ind_{L_{\delta^\mathbf{F}}}(u)
\]
that      is      $s_{\mathbf{F}}(C)$      is      of      the      form
$s_{\mathbf{F}}(C)(u)=(1/d!)\ind_{L_{\delta^\mathbf{F}}}(u)\prod_{i=1}^d
\bar{h}_i(u_{(i)})$   for   some   measurable   non-negative   functions
$(\bar{h}_i,   1\leq   i\leq   d)$.    Then,   thanks   to   Proposition
\ref{prop:c_delta_unique_product},          we         get          that
$S_{\mathbf{F}}(C)=C_{\delta^{\mathbf{F}}}$.    Then,  use   Proposition
\ref{prop:c_F-density} to get that $F=F_{\mathbf{F}}$.
\end{proof}

\section{Proofs} \label{sec:proof}

   \subsection{Preliminary notations for  the optimization problem} \label{sec:appendix0}
Recall notations from Sections \ref{sec:not-def} and
\ref{sec:connection}. In particular if $u=(u_1, \ldots, u_d) \in I ^d$
then $u^{OS}=(u_{(1)}, \ldots,
u_{(d)})$ denote the ordered vector of $u$.

   In order to apply the technique established in
   \cite{borwein1994entropy}, we introduce the
 linear functional $\ca=(\ca_i, 1 \leq i \leq 2d) : L^1 ( I^d) \rightarrow L^1(I)^{2d}$ as, for
 $f \in L^1 ( I^d)$ and $r\in I$ :
 \begin{equation*}
     \ca_i(f)(r)  = \int_{I^d} f(u) \ind_{\{u_i \leq r\}} \, du
\quad \text{and}\quad
  \ca_{d+i}(f)(r)  = \int_{I^d} f(u) \ind_{\{u_{(i)} \leq r\}} \, du
  \quad\text{for}\quad  1 \leq i \leq d .
 \end{equation*}
Let $\delta=(\delta_{(i)}, 1\leq i\leq d)\in \cd^0$ be a multidiagonal,
see Definition \ref{defi:diag}.
We set  $b^\delta = (b_i, 1 \leq i \leq 2d) $ given by
$b_i=\text{id}_I$ the identity function on $I$
  and $b_{d+i} = \delta_{(i)}$, for
 $1 \leq i \leq d$. If, for $c\in L^1(I^d)$, we have  $\ca_i(c)=b_i$, $1
 \leq i \leq d$ and $c \geq 0$ a.e., then we deduce that $c$ is the density of an absolutely
 continuous copula, say $C$. If we further have
 $\ca_{d+i}(c)=b_{d+i}$, for $1 \leq i \leq d$, then $\delta$ is the
 multidiagonal of $C$.

\begin{lem} \label{lem:unif}
Let $\delta\in \cd^0$ and $b^\delta = (b_i, 1 \leq i \leq 2d) $. If $c\in
L^1(I^d)$ is non-negative,
symmetric and satisfies $\ca_{d+i}(c)=b_{d+i}$ for $1 \leq i \leq d$,
then $c$ is the density of a copula with multidiagonal $\delta$.
\end{lem}

\begin{proof}
   The symmetry and  non-negativity of $c$ as well as the condition
   $\ca_{d+1}(c)(1)=\int_{I^d} c= b_{d+1}(1)=1$
ensures that $c$ is a density function of an exchangeable random vector
$V=(V_1, \hdots, V_d)$ on $I^d$. Recall  $V^{OS}=(V_{(1)}, \hdots,
V_{(d)})$ denotes the corresponding order statistics. By symmetry, the
lemma is proved as soon as we check that  $\ca_{1}(c)=b_{1}$. We have
for $r\in I$:
\[
   \ca_{1}(c)(r)  = \P (V_1 \leq r )
                  = \sum_{i=1}^d \P (V_{(i)} \leq r | V_1 = V_{(i)}) \P(V_1 = V_{(i)})
                  = \sum_{i=1}^d \delta_{(i)}(r) \inv{d}
                  = r,
\]
where we used the exchangeability of $V$ and the definition of
$\delta_{(i)}$ for the third equality, and \reff{eq:sum_delta} for the
last. This gives  $\ca_{1}(c)=b_{1}$.
\end{proof}

\subsection{Proof of Proposition \ref{prop:c_delta}}  \label{sec:appendixA}
Let $\delta\in \cd^0$.
Lemma \ref{lem:delta'_Psi} implies that $\delta_{(i)}((\Psi_i^\delta)^c)$ has zero Lebesgue measure
for all $2 \leq i \leq d$ with $\Psi_i^\delta$ given by \reff{eq:def_Psi_i_F}. By construction, the function $c_\delta$ defined by \reff{eq:c_delta}
is non-negative, symmetric  and well defined a.e. on $I^d$.
Recall the notation $(g_i^{(j)}, d_i^{(j)})$ used in \reff{eq:def_alpha_beta}. We define the functions $B_i$ on $I$ as, for $1 \leq i \leq d+1$, $t \in (g_i^{(j)}, d_i^{(j)})$ (with the conventions $\Psi_{1}^\delta=(0,d_1),\Psi_{d+1}^\delta=(g_{d+1},1)$):
\begin{equation} \label{eq:Bd}
       B_{d+1}(t) =1
\quad \text{and}\quad
       B_i(t) = \int_t^{d_i^{(j)}} a_i(s) B_{i+1}(s) \, ds\quad \text{for}\quad  1
       \leq i \leq d.
\end{equation}

For $t \in (\Psi_i^\delta)^c$, we set $B_i(t)=0$.
Recall $K_i$ defined in \reff{eq:F} for $1\leq i\leq d+1$ with the
convention $K_{d+1}=0$. We show that $B_i$ can be simply expressed by $K_i$
on $ \Psi^\delta_i$.

\begin{lem} \label{lem:Bi}
 Let $1 \leq i \leq d+1$ and $t \in \Psi^\delta_i$. Then we have:
\begin{equation} \label{eq:Bi_exp}
       B_i(t)= \exp\left(-K_i(t) \right) .
\end{equation}
\end{lem}

\begin{proof}
 For $i=d+1$, the result is trivial. We proceed by induction on $i$. We suppose that $B_{i+1}(t)=\exp\left(-K_{i+1}(t) \right)$ holds for
  some $1\leq i \leq d$, and all $t \in \psi^\delta_{i+1}$. We have for $t \in (g_i^{(j)}, d_i^{(j)})$:
 \begin{align*}
   B_i(t) & = \int_t^{d_i^{(j)}} a_i(s) B_{i+1}(s) \, ds \\
          & = \int_t^{d_i^{(j)}}  K'_i(s) \expp{K_{i+1}(s)-K_i(s)} \ind_{ \Psi_i^\delta \cap \Psi_{i+1}^\delta}(s) B_{i+1}(s) \, ds \\
          & = \int_t^{d_i^{(j)}}  K'_i(s) \expp{-K_i(s)} \ind_{ \Psi_i^\delta \cap \Psi_{i+1}^\delta}(s) \, ds. \\
          & = \int_t^{d_i^{(j)}}  K'_i(s) \expp{-K_i(s)} \, ds \\
          & = \exp(-K_i(t)),
 \end{align*}
 where we  used the definition  of $a_i$  given by \reff{eq:ai}  for the
 second  equality,  the induction  hypothesis  for  the third  equality,
 $(t,d_i^{(j)})  \subset \Psi_i^\delta$  and Lemma  \ref{lem:delta'_Psi}
 for the  fourth equality, and  finally Remark \ref{rem:d^i_t-}  for the
 fifth equality. This ends the induction.
\end{proof}

Similarly, we define the  functions $E_i$ on $I$ as, for $0\leq i \leq d$ as for $t \in (g_{i+1}^{(j)}, d_{i+1}^{(j)})$:
\begin{equation} \label{eq:Ed}
       E_0(t) =1,
\quad \text{and}\quad
       E_i(t) = \int_{g_{i+1}^{(j)}}^t a_i(s) E_{i-1}(s) \, ds\quad \text{for}\quad  1
       \leq i \leq d.
\end{equation}
 For $t \in  (\Psi_{i+1}^\delta)^c$ we set $E_i(t)=0$.
 The next Lemma gives a simple formula for $E_i$ on  $\Psi_{i+1}^\delta$.
 \begin{lem} \label{lem:Ei}
  Let $0 \leq i \leq d$ and $t \in \Psi^\delta_{i+1}$. Then we have:
  \begin{equation} \label{eq:Ci_exp}
       E_i(t)= \left(\delta_{(i)}(t)-\delta_{(i+1)}(t)\right) \exp\left(K_{i+1}(t) \right) .
\end{equation}

 \end{lem}

\begin{proof}
 For $i=0$ the result is clear thanks to the convention $\delta_{(0)}=1$ and \reff{eq:K1}. We proceed by induction on $i$. We suppose that $E_{i-1}(t)=(\delta_{(i-1)}(t)-\delta_{(i)}(t)) \exp(K_{i}(t))$ holds for
  some $1 \leq i \leq d$, and all $t \in \psi^\delta_{i}$. Let us denote $h_i=\delta_{(i-1)}-\delta_{(i)}$. Before computing $E_i(t)$ for $t \in (g_{i+1}^{(j)}, d_{i+1}^{(j)})$,
 we give an alternative expression for $\exp(K_i(s))$ for  $s \in (g_{i+1}^{(j)}, t)$:
 \begin{align} \label{eq:exp_Ki}
 \nonumber  \expp{K_{i+1}(s)} & = \exp\left( - \int_{m_{i+1}^{(j)}}^t \frac{h'_{i+1}(u)}{h_{i+1}(u)}+ \int_{s}^t \frac{h'_{i+1}(u)}{h_{i+1}(u)}+\int_{m_{i+1}^{(j)}}^s \frac{\delta'_{(i)}(u)}{h_{i+1}(u)}  \, du\right) \\
                    & = \frac{h_{i+1}(t)}{h_{i+1}(s)} \exp\left( - \int_{m_{i+1}^{(j)}}^t \frac{h'_{i+1}(u)}{h_{i+1}(u)}+\int_{m_{i+1}^{(j)}}^s \frac{\delta'_{(i)}(u)}{h_{i+1}(u)}  \, du\right).
 \end{align}
 Then we have for $t \in (g_{i+1}^{(j)}, d_{i+1}^{(j)})$:
 \begin{align*}
    E_{i}(t) & = \int_{g_{i+1}^{(j)}}^t a_i(s) E_{i-1}(s) \, ds \\
             & = \int_{g_{i+1}^{(j)}}^t K'_i(s) \expp{K_{i+1}(s)-K_i(s)} \ind_{ \Psi_i^\delta \cap \Psi_{i+1}^\delta}(s) E_{i-1}(s) \, ds \\
             & = \int_{g_{i+1}^{(j)}}^t K'_i(s) \expp{K_{i+1}(s)} h_{i}(s)\ind_{ \Psi_i^\delta \cap \Psi_{i+1}^\delta}(s) \, ds \\
             & = \int_{g_{i+1}^{(j)}}^t \delta'_{(i)}(s) \expp{K_{i+1}(s)} \ind_{ \Psi_i^\delta \cap \Psi_{i+1}^\delta}(s) \, ds \\
             & = h_{i+1}(t) \exp\left( - \int_{m_{i+1}^{(j)}}^t \frac{h'_{i+1}(u)}{h_{i+1}(u)} \, du \right)  \int_{g_{i+1}^{(j)}}^t \left(\frac{\delta'_{(i)}(s)}{h_{i+1}(s)} \exp\left(\int_{m_{i+1}^{(j)}}^s \frac{\delta'_{(i)}(u)}{h_{i+1}(u)}  \, du\right) \right) \, ds  \\
             & = h_{i+1}(t) \exp\left( - \int_{m_{i+1}^{(j)}}^t \frac{h'_{i+1}(u)-\delta'_{(i)}(u)}{h_{i+1}(u)}  \, du\right)\\
             & = h_{i+1}(t) \exp(K_{i+1}(t)),
 \end{align*}
where we used the definition of $a_i$ given by \reff{eq:ai} for the
second equality, the induction hypothesis for the third equality,
Lemma \ref{lem:delta'_Psi} and \reff{eq:exp_Ki} for the fifth equality, and for the seventh equality we use that, for $t \in (g_{i+1}^{(j)}, m_{i+1}^{(j)})$ (similarly to Remark \ref{rem:d^i_t-}):
 \begin{equation*}
    \int_{m_{i+1}^{(j)}}^t \frac{\delta'_{(i)}(s)}{h_{i+1}(s)}  \, ds \leq  \int_{m_{i+1}^{(j)}}^t \frac{\delta'_{(i)}(s)}{\delta_{(i)}(s) - \delta_{(i+1)}(g_{i+1}^{(j)}) }  \, ds = \log\left(\frac{\delta_{(i)}(t)-\delta_{(i+1)}(g_{i+1}^{(j)})}{\delta_{(i)}(m_{i+1}^{(j)})-\delta_{(i+1)}(g_{i+1}^{(j)})}\right),
 \end{equation*}
giving $\lim_{t \searrow g_{i+1}^{(j)}} \int_{m_{i+1}^{(j)}}^t \frac{\delta'_{(i)}(s)}{h_{i+1}(s)}  \, ds = -\infty$.

\end{proof}

The following Lemma justifies the introduction of the functions $B_i, E_i$.

\begin{lem} \label{lem:B_iE_i}
  We have with $u_{(0)}=0$ for $1 \leq i \leq d$, $t \in \Psi_i^\delta$:
  \begin{equation} \label{eq:B_iE_i}
     \int_{I^d} c_\delta(u) \ind_{\{u_{(i-1)} \leq t \leq u_{(i)}\}} \, du = B_i(t) E_{i-1}(t).
  \end{equation}
\end{lem}

\begin{proof}
  The definition \reff{eq:Bd} of $B_i$ for $1 \leq i \leq d$ gives that for $t\in I$:
  \[
    B_i (t) =  \int  a_i(r_i) a_{i+1}(r_{i+1})\hdots a_d(r_d) \ind_{\{t
      \leq r_i \leq r_{i+1} \leq \hdots \leq r_d \leq 1\}}
    \ind_{\{[t,r_i) \subset \Psi_{i}^\delta \}}\prod_{j=i}^{d-1}
    \ind_{\{(r_j,r_{j+1}) \subset \Psi_{j+1}^\delta \}} \, dr,
  \]
  with $r=(r_i,r_{i+1}, \hdots , r_d) \in I^{d-i+1}$. Similarly, we have
  for $1\leq i \leq d$, $t \in I$ that $E_{i-1} (t)$ is equal to:
  \[
 \int  a_1(q_1)a_{2}(q_{2})\hdots a_{i-1}(q_{i-1})
    \ind_{\{ 0 \leq q_1 \leq q_{2} \leq \hdots \leq q_{i-1} \leq t\}}
    \ind_{\{(q_{i-1},t] \subset \Psi_{i}^\delta \}}\prod_{j=1}^{i-2}
    \ind_{\{(q_j,q_{j+1}) \subset \Psi_{j+1}^\delta \}} \, dq,
  \]
  with $q=(q_1,q_2, \hdots, q_{i-1}) \in I^{i-1}$. Multiplying $B_i(t)$
  with  $E_{i-1}(t)$ gives:
  \begin{align*}
   B_i (t) E_{i-1} (t) & = \int_\triangle \prod_{j=1}^d a_j(u_j) \ind_{\{u_{i-1} \leq t \leq u_{i}\}} \prod_{j=1}^{d-1} \ind_{\{ (u_j,u_{j+1}) \subset \Psi_{j+1}^\delta\}} \, du \\
                       & = \int_\triangle \prod_{j=1}^d a_j(u_j) \ind_{\{u_{i-1} \leq t \leq u_{i}\}} \ind_{L^\delta}(u) \, du \\
                       & = d! \int_\triangle c_\delta(u) \ind_{\{u_{i-1} \leq t \leq u_{i}\}} \, du \\
                       & = \int_{I^d} c_\delta(u) \ind_{\{u_{(i-1)} \leq t \leq u_{(i)}\}} \, du,
  \end{align*}
  where we used the symmetry of $c_\delta$ for the fourth equality.
\end{proof}

Lemma \ref{lem:B_iE_i} with $i=1$ ensures that $\int_{I^d} c_\delta(u) \, du = \lim_{t \searrow 0} B_1(t)E_0(t)=1$, that is $c_\delta$ a probability density function on $I^d$.
Now we compute $\ca_{d+1}(c_\delta)$. We have, for  $t \in \Psi_i^\delta$:
\begin{equation*}
      \ca_{d+1}(c_\delta)(t) = \int_{I^d} c_\delta(u) \ind_{\{u_{(1)} \leq t\}} \, du = 1 - \int_{I^d} c_\delta(u) \ind_{\{u_{(1)} \geq t\}} \, du = 1- B_1(t)E_0(t) = \delta_{(1)}(t),
\end{equation*}
    where we used Lemma \ref{lem:B_iE_i} with $i=1$ for the third equality, then \reff{eq:Bi_exp}  and \reff{eq:K1}  for the fourth equality. By continuity this gives $\ca_{d+1}(c_\delta)=\delta_{(1)}$ on $I$. For $2 \leq i \leq d$, we have by induction for $t \in \Psi_i^\delta$:
\begin{align*}
     \ca_{d+i}(c_\delta)(t) & = \int_{I^d} c_\delta(u) \ind_{\{u_{(i)} \leq t\}} \, du \\
			 & = \int_{I^d} c_\delta(u) \ind_{\{u_{(i-1)} \leq t\}} \, du - \int_{I^d} c_\delta(u) \ind_{\{u_{(i-1)} \leq t \leq u_{(i)}\}} \, du \\
			 & =  \ca_{d+i-1}(c_\delta)(t) - B_i(t)E_{i-1}(t) \\
			 & =\delta_{(i-1)}(t)-\left(\delta_{(i-1)}(t)-\delta_{(i)}(t)\right) \\
			 & =\delta_{(i)}(t),
\end{align*}
where we  used the  induction and Lemma  \ref{lem:B_iE_i} for  the third
equality,  as  well as  \reff{eq:Bi_exp}  and  \reff{eq:Ci_exp} for  the
fourth. By continuity, we  obtain $\ca_{d+i}(c_\delta)= \delta_{(i)}$ on
$I$. Then use Lemma \ref{lem:unif} to get that $c_\delta$ is the density
of a (symmetric) copula, say $C_\delta$, with multidiagonal $\delta$.

   The computation of $H(C_\delta)$ is technical.
    By symmetry of $C_\delta$, we have:
    \[
      H(C_\delta)=-d! \int_\triangle c_\delta \log(c_\delta).
    \]

    Let $\varepsilon\in (0, 1/2)$.  We define  for $1 \leq i \leq d$ the
    set  $I^\varepsilon_i=\{  t  \in I;  (t-\varepsilon,  t+\varepsilon)
    \subset  \Psi^\delta_i   \cap  \Psi^\delta_{i+1}  \}$  as   well  as
    $I^\varepsilon_0=(\varepsilon,         d_1-\varepsilon)$         and
    $I^\varepsilon_{d+1}=(g_{d+1}+\varepsilon,1-\varepsilon)$.   We  set
    $\triangle^\varepsilon     =     \triangle    \cap     \prod_{i=1}^d
    I_i^\varepsilon$.   Since $x\log(x)\geq  -1/\expp{}$  for $x>0$,  we
    deduce by monotone convergence that:
\begin{equation}
   \label{eq:Ice}
H(C_\delta)=\lim_{\varepsilon\downarrow 0} H_\varepsilon(C_\delta),
\end{equation}
with:
\[
H_\varepsilon(C_\delta)=
-d! \int_{\triangle^\varepsilon} c_\delta\log(c_\delta).
\]
We can decompose $H_\varepsilon(C_\delta)$ as:
\begin{equation}
\label{eq:decoupe}
 H_\varepsilon(C_\delta)= J_1(\varepsilon)  + J_2(\varepsilon)
 +J_3(\varepsilon)  + J_4(\varepsilon) ,
\end{equation}
with:
\begin{align*}
 J_1(\varepsilon) & = - d! \sum_{i=1}^d \int_{\triangle^\varepsilon} c_\delta(u)  \log(\delta_{(i)}'(u_i)) \, du , \\
 J_2(\varepsilon) & =  d!\sum_{i=2}^d  \int_{\triangle^\varepsilon}  c_\delta(u) \log \left(
   \delta_{(i-1)}(u_i) -\delta_{(i)}(u_i) \right)\,du , \\
 J_3(\varepsilon) & = d!\sum_{i=2}^d  \int_{\triangle^\varepsilon} c_\delta(u) \left( K_i(u_i)- K_i(u_{i-1}) \right) \,du , \\
 J_4(\varepsilon) & = \log(d!) \ d! \int_{\triangle^\varepsilon} c_\delta(u) \, du ,
\end{align*}
where we used \reff{eq:K1} to express $a_1=\delta'_{(1)}\expp{K_2}$ a.e.,
 so that the sums in $J_2$ and $J_3$ start
at $i=2$.

Since $\ca_{d+i}(c_\delta)=\delta_{(i)}$, we deduce that for $1\leq
i\leq d$ and any measurable non-negative function $h$ defined on $I$:
\begin{equation}
   \label{eq:int-c-h}
\int_{I^d} c_\delta(u) h(u_{(i)})\, du=\int_I \delta'_{(i)}(t) h(t)\,
dt.
\end{equation}
In particular, we get:
\[
\int_{I^d} c_\delta(u)\val{\log(\delta'_{(i)}(u_{(i)}))} \, du
=\int_I \delta'_{(i)}(t)\val{\log(\delta'_{(i)}(t))} \,
dt,
\]
which is finite thanks to Remark \ref{rem:dlogd}.
Therefore, by the dominated convergence and using
\reff{eq:int-c-h} again, we have:
\[
\lim_{\varepsilon \rightarrow 0} J_1(\epsilon) =
\sum_{i=1}^d H(\delta_{(i)}).
\]

For $J_2(\varepsilon)$, we notice that the integrand is non-positive a.e., since for $t\in I$, $2 \leq i \leq d$, we have
$\delta_{(i-1)}(t) -\delta_{(i)}(t) \leq 1$. Therefore, we get:
\[
 \lim_{\varepsilon \rightarrow 0} J_2(\epsilon) = -d! \sum_{i=2}^d
 \int_{ \triangle} c_\delta(u)  \val{\log (\delta_{(i-1)}(u_i)
   -\delta_{(i)}(u_i) )} \, du=-\cj(\delta),
\]
where we used  monotone convergence for the first equality and \reff{eq:int-c-h}
as well as the definition \reff{eq:cj_delta} of $\cj(\delta)$ for the second.

We now compute $J_3(\varepsilon)$.   For $t \in I^\varepsilon_{d+1}$, we
set $ B^{\varepsilon}_{d+1}(t)=  1$, and for $t  \in I^\varepsilon_0$, $
E^{\varepsilon}_0(t) =1$.  For  $1 \leq i \leq d$, we  define for $t \in
(g_i^{(j)},d_i^{(j)})$ and $t' \in (g_{i+1}^{(j)},d_{i+1}^{(j)})$:
    \[
       B^{\varepsilon}_i(t) = \int_t^{\max(d_i^{(j)}-\varepsilon,t) } a_i(s) B^{\varepsilon}_{i+1}(s) \, ds,
       \quad \text{ and } \quad E^{\varepsilon}_i(t') =
       \int_{\min(g_{i+1}^{(j)}+\varepsilon,t')}^{t'} a_i(s)
       E^{\varepsilon}_{i-1}(s) \, ds.
    \]
By  monotone convergence, we have for $t \in \Psi^\delta_i$ and  $t' \in \Psi^\delta_{i+1}$:
\[
  \lim_{\varepsilon \rightarrow 0} B^{\varepsilon}_i(t) = B_i(t) \quad \text{ and } \quad
  \lim_{\varepsilon \rightarrow 0} E^{\varepsilon}_i(t') = E_i(t').
\]
An integration by parts  gives:
\begin{align*}
J_3(\varepsilon)
& = d! \sum_{i=2}^d \left(\int_{\triangle^\varepsilon} c_\delta(u)
  K_i(u_i) \, du  -\int_{\triangle^\varepsilon} c_\delta(u) K_i(u_{i-1})
  \, du \right)  \\
& =  \sum_{i=2}^d \left(\int_{I_i^\varepsilon} E_{i-1}^{\varepsilon} (t)
  a_i(t)K_i(t) B_{i+1}^{\varepsilon}(t) \, dt
  -\int_{I_{i-1}^\varepsilon}  E_{i-2}^{\varepsilon} (t)
  a_{i-1}(t)K_i(t) B_{i}^{\varepsilon}(t)\, dt \right) \\
& =  \sum_{i=2}^d \left(\int_{I_i^\varepsilon} E_{i-1}^{\varepsilon} (t)
  a_i(t)K_i(t) B_{i+1}^{\varepsilon}(t) \, dt  \right) \\
& \hspace{2cm} +  \sum_{i=2}^d \left( \int_{I_i^\varepsilon}
  E_{i-1}^{\varepsilon} (t) K'_i(t)  B_{i}^{\varepsilon}(t)\, dt   -
  \int_{I_i^\varepsilon} E_{i-1}^{\varepsilon} (t) K_i(t) a_i(t)
  B_{i+1}^{\varepsilon}(t) \, dt  \right) \\
& =   \sum_{i=2}^d \int_{I_i^\varepsilon} E_{i-1}^{\varepsilon} (t)
K'_i(t)  B_{i}^{\varepsilon}(t)\, dt.
\end{align*}
   By  monotone convergence and thanks to \reff{eq:F}, \reff{eq:Bi_exp}
   and \reff{eq:Ci_exp}, we get:
\[
    \lim_{\varepsilon \rightarrow 0} J_3(\varepsilon)
 =
         \sum_{i=2}^d \int_{I} E_{i-1} (t) K'_i(t)
         B_{i}(t)\, dt
=  \sum_{i=2}^d \int_{I} \delta_{(i)}'(t) \, dt
=  (d-1) .
\]
We deduce from \reff{eq:Ice}
and \reff{eq:decoupe} and the limits for $J_1$, $J_2$, $J_3$
and $J_4$ as  $\varepsilon $ goes down to 0 that:
\[
 H(C_\delta) = \sum_{i=1}^d H(\delta_{(i)})-\cj(\delta) + (d-1) + \log(d!).
\]

 \subsection{The optimization problem}  \label{sec:appendixB}
Let $\delta\in \cd^0$. Recall notation from Section \ref{sec:appendix0}. The problem of maximizing $H$ over
$\cc_\delta^0$ can be written as
 an optimization problem $(P^\delta)$ with infinite dimensional
 constraints:
 \begin{equation}
\tag{$P^{\delta}$}
\text{maximize } H(c) \text{ subject to }
\begin{cases}
 &  \ca(c)=b^\delta,\\
&c\geq 0 \text{ a.e. and } c\in L^1(I^d).
\end{cases}
\end{equation}

Notice   that   if   $f\in   L^1(I^d)$  is   non-negative   and   solves
$\ca(f)=b^\delta$, then $f$  is the density of a copula.
We say that a
function $f$ is  feasible for $(P^{\delta})$ if  $f\in L^1(I^d)$, $f\geq
0$ a.e.,  $\ca(f)=b^\delta$ and $H(f)>-\infty $.  We say that $f$  is an
optimal  solution of  $(P^{\delta})$ if  $f$ is  feasible and  $H(f)\geq
H(g)$ for all $g$ feasible.
The next Proposition gives conditions which ensure the existence of an
optimal solution.

\begin{prop}
   \label{prop:sym}
   Let  $\delta\in \cd^0$.  If there  exists $c$  feasible  for $(P^\delta)$, then  there
   exists  a  unique  optimal  solution  to  $(P^{\delta})$  and  it  is
   symmetric.
\end{prop}

\begin{proof}
Since $\ca(f)=b^\delta$ implies $\ca_1(f)(1)=b_1(1)$ that is $ \int_{I^d}f(x)\,
dx=1$,  we can directly apply Corollary~2.3 of \cite{borwein1994entropy} which
states that if there exists a feasible $c$, then there exists a
unique optimal solution to $(P^{\delta})$. Since the constraints of
$(P^{\delta})$ are symmetric, such as the functional $H$, we deduce that
if $c^*$ is the optimal solution, then so is $c^*_\pi$ defined for $\pi
\in \cs_d$ and $u\in I^d$ as $  c^*_\pi(u) = c^*(u_{\pi})$.
 By uniqueness of the optimal solution, we deduce that $c^*=c^*_\pi$ for all permutations $\pi \in \cs_d$; hence $c^*$
 is symmetric.
\end{proof}

Combining Lemmas \ref{lem:c_zero_Psi} and  \ref{lem:0} gives the following
Corollary on the support of any $c$ verifying $\ca(c)=b^\delta$.
\begin{cor} \label{cor:ZL}
  Let  $\delta\in \cd^0$.  If   $c\in L^1(I^d)$ is non-negative and
   verifies $\ca(c)=b^\delta$,  then $c=0$ a.e. on $Z_\delta \bigcup L_\delta^c$ with
   $L_\delta$ defined by \reff{eq:def_LD} and
   $ L_\delta^c=I^d \setminus  L_\delta$.
\end{cor}

\subsection{Reduction of the optimization problem $(P^\delta)$}
Let $\delta\in \cd^0$.
Since the optimal solution of $(P^\delta)$ is symmetric, see Proposition
\ref{prop:sym}, we can
reduce the optimization problem by considering it on the simplex $\triangle$.
We define $\mu$ to be the Lebesgue measure restricted to $\left(Z_\delta^c \cap L_\delta \right) \cap \triangle$:
$\mu(du) = \ind_{\left(Z_\delta^c \cap L_\delta \right) \cap \triangle} (u) du$. We define, for $f\in L^1(I^d)$:
\[
 H^\mu(f)= -\int_{I^d} f(u) \log(f(u)) \, \mu(du).
\]
From Corollary \ref{cor:ZL} we can deduce that if $c\in L^1(I^d)$ is
non-negative symmetric and solves $\ca(c)=b^\delta$, then:
\begin{equation} \label{eq:ct}
H(c)= d! \, H^\mu(c).
\end{equation}
Let us also define, for $f\in L^1(I^d)$,  $1 \leq i \leq d$, $r \in I$:
\[
  \ca^\mu_i(c)(r) = d!\, \int_{I^d} c(u) \ind_{\{u_{i} \leq r \}} \, \mu(du).
\]
We shall consider the restricted optimization problem $(P^\delta_\mu)$ given by:
 \begin{equation}
\tag{$P^{\delta}_\mu$}
\text{maximize } H^\mu(c) \text{ subject to }
\begin{cases}
 &  \ca^\mu(c)=\delta,\\
&c\geq 0 \text{ $\mu$-a.e. and } c\in L^1(I^d).
\end{cases}
\end{equation}
We have the following equivalence between ($P^\delta$) and
($P^\delta_\mu$). Recall $u^{OS}$ denote the ordered vector of $u\in
\R^d$.
\begin{cor}
   \label{cor:c-triangle}
   Let $\delta\in \cd^0$.  If $c$ is the optimal  solution of ($P^\delta$)
   then it  is also an optimal  solution to
   ($P^\delta_\mu$).  If  $\hat c$  is  an  optimal solution  of
   ($P^\delta_\mu$),  then  $c$,  defined  by  $c(u)=\hat c(u^{OS})\ind_{Z_\delta^c \cap L_\delta}(u)$  is
   the optimal solution to ($P^\delta$).
\end{cor}
Notice the Corollary  implies that ($P^\delta_\mu$) has  a $\mu$-a.e.
unique optimal solution: if $c_1$ and $c_2$ are two optimal
solutions of ($P^\delta_\mu$) then $\mu$-a.e. $c_1=c_2$.
 Thanks to    Proposition \ref{prop:sym}  and
\reff{eq:ct},  Corollary \ref{cor:c-triangle} is a direct consequence of
the following Lemma that establishes the connection between the
constraints.

\begin{lem}
\label{lem:reduction}
  Let $\delta \in  \cd^0$. For $c \in L^1(I^d)$ symmetric and non-negative the following two conditions are
  equivalent:
  \begin{enumerate}
   \item $\ca(c)=b^\delta$.
   \item $\ca^\mu(c)=\delta$ and $c=0$ a.e. on $Z_\delta \bigcup L_\delta^c$.
  \end{enumerate}
\end{lem}

\begin{proof}
  Assume that $\ca(c)=b^\delta$. We have, by Corollary \ref{cor:ZL}, that $c=0$ a.e. on $Z_\delta \cup L_\delta^c$.
  This and the symmetry of $c$ gives,
  for $1 \leq i \leq d$, $r \in I$:
\[
     \ca^\mu_i(c)(r)
                      = d! \int_{I^d} c(u) \ind_{\{u_{(i)} \leq r \}}
                     \ind_{\triangle} (u) \, du
                      = \int_{I^d} c(u) \ind_{\{u_{(i)} \leq r \}} \, du
                      =  \delta_{(i)}(r) .
\]
  On the other hand, let us assume that $\ca^\mu(c)=\delta$ and $c=0$ a.e. on $Z_\delta \cup L_\delta^c$. We have, for
  $1\leq i \leq d$, $r \in I$:
\[
     \ca_{d+i}(c)(r)  = \int_{I^d} c(u) \ind_{\{u_{(i)} \leq r \}}
     \ind_{Z_\delta^c \cap L_\delta} (u) \, du
                      = d! \int_{I^d} c(u) \ind_{\{u_{i} \leq r \}}
\, \mu(du  )
                       = \delta_{(i)}(r),
\]
  where we used $c=0$ a.e. on $Z_\delta \cup L_\delta^c$ for the first equality, the
  symmetry of $c$ and the definition of $\mu$  for the second, and $\ca^\mu(c)=\delta$ for
  the third.
  Lemma \ref{lem:unif} ensures then that $\ca_{i}(c)=b_i$ for $1 \leq i \leq d$. This ends the proof.
\end{proof}

\subsection{Solution for the reduced optimization problem
  $(P^\delta_\mu)$}

Let $\delta\in \cd^0$. We compute $(\ca^\mu)^*: L^\infty (I)^{d} \rightarrow L^\infty (I^d)$ the adjoint of $\ca^\mu$.
For $\lambda=(\lambda_i, 1\leq i \leq d) \in L^{\infty} (I)^{d} $ and $f\in
 L^1 (I^d)$, we have:
\[
\langle (\ca^\mu)^*(\lambda),f \rangle
= \langle \lambda,\ca^\mu(f)  \rangle
= \sum_{i=1}^d \int_I  \lambda_i (r)\!\!\int_{I^d}\! f(u)\ind_{\{u_{i}\leq r\}} d \mu(u) \, dr
=\int_{I^d}  \! f(u) \sum_{i=1}^d \Lambda_i(u_i) \, d\mu(u),
\]
where we used  the definition of the adjoint
operator for the first equality, Fubini's theorem for the second, and the
following definition of the functions $(\Lambda_i, 1\leq i\leq d)$ for
the third:
\[
\Lambda_i(t)= \int_I  \lambda_i(r)\ind_{\{r\geq t\}}\,  dr, \quad t\in I.
\]
Thus, we have for $\lambda \in L^\infty(I)^{d}$ and $u=(u_1, \ldots, u_d) \in I ^d$:
\begin{equation}
   \label{eq:L*}
(\ca^\mu)^*(\lambda)(u)=\sum_{i=1}^d \Lambda_i(u_i).
 \end{equation}

 We will use Theorem 2.9. from \cite{borwein1994entropy} on abstract
 entropy minimization, which we recall here, adapted to the context of
 $(P^{\delta}_\mu)$.

\begin{theo}[Borwein, Lewis and Nussbaum] \label{theo:bor_lew_nuss}
 Suppose there exists  $ c>0$ $\mu$-a.e. which is feasible for
 $(P^{\delta}_\mu)$. Then there exists a $\mu$-a.e. unique optimal solution,
 $ c ^*$, of $(P^{\delta}_\mu)$. Furthermore, we have
 $ c^*>0$ $\mu$-a.e. and there exists a sequence $(\lambda^n,
 n\in \N^*)$ of elements of $L^\infty (I)^{d}$ such that:
\begin{equation}
   \label{eq:cvA}
\int_{I^d}  c^*(u) \left| (\ca^\mu)^*(\lambda^n)(u)
  -\log( c^*(u)) \right| \; \mu(du)
\; \xrightarrow[n\rightarrow \infty ]{} \;0.
\end{equation}
\end{theo}

Now  we  are  ready  to  prove   that  the  optimal  solution  $c^*$  of
$(P^\delta_\mu)$ is the product of measurable univariate functions.

\begin{lem}
 \label{lem:c*=ab}
 Let $\delta\in \cd^0$.
 Suppose that there exists $c>0$ $\mu$-a.e.which is feasible for $(P^\delta_\mu)$.
 Then there exist  non-negative, measurable functions $(a^*_i$, $1 \leq i \leq d)$
 defined on $I$ such that $a^*_i(s)=0$ if $ \delta_{(i)}'(s)=0$ and the
 function $c^*$ defined a.e. on $I^d$ by:
 \[
 c^*(u)=\inv{d!} \ind_{L_\delta}(u) \prod_{i=1}^d  a^*_i(u_i)
 \]
 is the optimal solution to $(P^\delta_\mu)$.
\end{lem}

\begin{proof}

 According to Theorem \ref{theo:bor_lew_nuss},
there exists a sequence $(\lambda^n, n\in \N)$ of elements of
$L^\infty (I)^{d}$ such that the optimal solution, say $c^*$,  satisfies
\reff{eq:cvA}. This implies, thanks to \reff{eq:L*},
 that there exist $d$ sequences $(\Lambda_i^n, n\in \N^*, 1 \leq i \leq d)$
 of elements of $L^\infty (I)$ such that the
 following convergence holds in $L^1(I^d, c^* \mu)$:
\begin{equation}
   \label{eq:cvln}
 \sum_{i=1}^d\Lambda^{n}_i(u_i)
\; \xrightarrow[n\rightarrow \infty ] \;
\log(c^*(u)).
\end{equation}

  We first assume that there exist
$\Lambda_i$, $1 \leq i \leq d$  measurable functions defined on $I$ such that
 $\mu$-a.e. on $S$:
 \begin{equation}
  \label{eq:Gam_lim}
  \sum_{i=1}^{d} \Lambda_i(u_i) = \log(c^*(u)).
  \end{equation}
   Set $a^*_i=\sqrt[d]{d!}\exp(\Lambda_i)$ so that $\mu$-a.e. on $S$:
    \begin{equation}
     \label{eq:c=ab}
     c^*(u)=\inv{d!}\prod_{i=1}^{d} a^*_i(u_i).
    \end{equation}
Recall $\mu(du)=\ind_{(Z_\delta^c \cap L_\delta) \cap \triangle}  (u) \,
du$. From the definition \reff{eq:def-Z} of $Z_\delta$, we deduce that
without loss  of   generality,  we   can  assume   that   $a^*_i(u_i)=0$  if
$\delta_{(i)}'(u_i)=0$. Therefore we obtain $c^*(u)=(1/d!)\ind_{L_\delta}(u)\prod_{i=1}^{d} a^*_i(u_i)$ for $u \in I^d$.

To complete the proof, we now show that \reff{eq:Gam_lim} holds for $\Lambda_i$, $1 \leq i \leq d$ measurable functions.
We introduce the notation
 $u_{(-i)}=(u_1,\hdots, u_{i-1},u_{i+1},\hdots, u_d) \in I^{d-1}$.
 Let us define the probability measure $P(du)=c^*(u)\mu(du)/ \int_{I^d} c^*(y)\mu(dy)$
 on $I^d$. We fix $j$, $1 \leq j \leq d$. In order to apply Proposition 2 of \cite{Ruschendorf1993369}, which ensures the existence of the limiting measurable functions $\Lambda_i$, $1\leq i \leq d$,
 we first check that $P$ is absolutely continuous with respect to $P^j_1 \otimes P^j_2$,
 where $P^j_1(du_{(-j)})= \int_{u_j \in I} P(du_{(-j)} du_j)$ and $P^j_2(du_j)= \int_{u_{(-j)} \in I^{d-1}} P(du_{(-j)} du_j)$
 are the marginals of $P$. Notice that there exists a non-negative density function $h$ such that $P(du)=h(u_{(-j)},u_j)du_{(-j)}du_j$. Let $h_1(u_{(-j)})=\int h(u_{(-j)},u_j) d u_j$ and $h_2(u_{j})=\int h(u_{(-j)},u_j) d u_{(-j)}$ denote the density of the marginals $P^j_1$ and $P^j_2$. Then the density of the product measure $P^j_1 \otimes P^j_2$ is given by $P^j_1 \otimes P^j_2(du) = h_1(u_{(-j)})h_2(u_{j}) du_{(-j)}du_j$. The support of the density $h$ is noted by $T_0=\{u \in I^d; h(u) >0\}$, and the support of the marginals are noted by $T_1= \{v \in I^{d-1}; h_1(v) > 0\}$ and $T_2 = \{ t \in I; h_2(t)>0\}$. With this notation, we have that a.e. $T_0 \subset T_1 \times T_2$ (that is $T_0 \cap (T_1 \times T_2)^c$ is of zero Lebesgue measure).  If $A \subset I^d$ is such that $\int \ind_{A}(u)  h_1(u_{(-j)})h_2(u_{j}) du_{(-j)}du_j =0$, then we also have $\int \ind_{A \cap (T_1 \times T_2)}(u)  h_1(u_{(-j)})h_2(u_{j}) du_{(-j)}du_j =0$. Since $ h_1 h_2$ is positive on $T_1 \times T_2$, this implies that $A \cap (T_1 \times T_2)$ has zero Lebesgue measure. Therefore we have:
 \begin{equation*}
  \int \ind_A(u) h(u) du  = \int \ind_{A \cap (T_1 \times T_2)} (u) h(u) du + \int \ind_{A \setminus (T_1 \times T_2)} (u) h(u) du =0,
 \end{equation*}
 since $h=0$ a.e. on $A \setminus (T_1 \times T_2)$. This proves that $P$ is absolutely continuous with respect to $P^j_1 \otimes P^j_2$.
Then according to Proposition 2 of \cite{Ruschendorf1993369}, \eqref{eq:cvln} implies that
there exist measurable functions $\Phi_j $ and
$\tilde{\Lambda}_j $ defined respectively on $I^{d-1}$ and $I$, such that $c^*\mu$-a.e. on $\triangle$:
\[
\log(c^*(u))=\Phi_j(u_{(-j)})+\tilde{\Lambda}_j(u_j).
\]
As    ${\mu}$-a.e. $c^*>0$,  this equality   holds $\mu$-a.e. on $S$.
Since we have such a representation for every $1 \leq j \leq d$, we can easily
verify that  $\log(c^*(u))=\sum_{i=1}^d \Lambda_i(u_i)$ $\mu$-a.e. with $\tilde{\Lambda}_j =
\Lambda_j$ up to an additive constant.

\end{proof}

\subsection{Proof of Proposition \ref{prop:c_delta_unique_product}}
 \label{sec:proof-prod}
Let $\delta \in \cd^0$. Recall that
$u^{OS}$ denotes the ordered vector of $u\in \R^d$.
Let $c$ be the density of a symmetric copula in $\R ^d$ such that
$\ca(c)=b^\delta$ and $c$ is of  product
form, that is, thanks to Corollary \ref{cor:c-triangle},
$c(u)=c^*(u^{OS})$ with
\[
c^*(u)= \inv{d!}\, \ind_{L_\delta}(u) \prod_{i=1}^d a^*_i(u_i), \quad u=(u_1, \ldots, u_d)\in \triangle,
\]
where $a^*_i$,  $1 \leq  i \leq d$  are measurable  non-negative functions
defined  on  $I$. In  this  section,  we  shall  prove that  $c$  equals
$c_\delta$  defined by  \reff{eq:c_delta}; that is, for all $1\leq i\leq
d$,  $a^*_i$ is a.e. equal, up
to a multiplicative constant,  to
$a_i$ defined
in  \reff{eq:ai}. This  will prove  Proposition
\ref{prop:c_delta_unique_product}.

Recall the definitions of $g_i^{(j)},m_i^{(j)},d_i^{(j)}$ from  Section
\ref{sec:max_ent_copula}, for $1\leq i \leq d+1$.
We deduce from \reff{eq:L_delta_leq_prod} that:
\[
c^*(u)=\inv{d!} \,\ind_{L_\delta}(u) \prod_{i=1}^d a^*_i(u_i)\ind_{\Psi^\delta_i \cap \Psi^\delta_{i+1}}(u_i), \quad
u=(u_1, \ldots, u_d)\in \triangle.
\]
We deduce also from Lemma \ref{lem:reduction} that $\ca^\mu(c^*)=\delta$.
We introduce the following family of functions:
    \[
       B^*_{d+1}(t) =  E^*_0(t) = 1,
    \]
    and for $1 \leq i \leq d$, $t \in (g_i^{(j)}, d_i^{(j)})$ and  $t' \in (g_{i+1}^{(j)}, d_{i+1}^{(j)})$:
    \[
       B^*_i(t) = \int_t^{d_i^{(j)}} a^*_i(s) B^*_{i+1}(s) \, ds,
       \quad E^*_i(t') = \int_{g_{i+1}^{(j)}}^{t'} a^*_i(s) E^*_{i-1}(s) \, ds.
    \]
Recall the functions $B_i$, for $1\leq i\leq d+1$,  and $E_i$, for $0\leq
i\leq  d$  defined by \reff{eq:Bd} and
\reff{eq:Ed}.
We will prove by (downward) induction on $i\in\{1, \ldots, d+1\}$ that:
\begin{equation}
   \label{eq:HRB}
     B^*_i(t)=B^*_i(m_i^{(j)}) B_i(t), \quad t \in (g_i^{(j)}, d_i^{(j)}) .
\end{equation}
For $i=d+1$, it trivially holds.  Let us assume that \reff{eq:HRB} holds
for  $i+1$,  $d\geq  i  \geq 1$.   Recall  the  convention  $K_{d+1}=0$,
$\delta_{(d+1)}=0$          and         $\delta_{(0)}=1$. Arguing as in the proof of Lemma \ref{lem:B_iE_i},   we    deduce  from
$\ca^\mu_{i}(c^*)=\delta_{(i)}$ that   for    $r \in \Psi_i^\delta$:
   \begin{align*}
  \delta_{(i)}(r)
&= {d!} \int_{I^d} c^*(u)\ind_{\{ u_{(i)} \leq r\}} \, \mu(du) \\
  & = {d!} \int_{I^d} c^*(u)\ind_{\{ u_{(i+1)} \leq r\}} \, \mu(du)  + \int_{\triangle}  \ind_{L_\delta}(u) \prod_{j=1}^d  \left(a^*_j(u_j) \ind_{ \Psi^\delta_j \cap \Psi^\delta_{j+1} }(u_j) \right) \ind_{
    \{u_{i} \leq r \leq u_{i+1} \} }  \, du \\
  & = \delta_{(i+1)}(r) + B^*_{i+1}(r) E^*_{i}(r).
   \end{align*}
 This gives on $\Psi_i^\delta$:
\begin{equation}
   \label{eq:BE*}
 \delta_{(i)} -\delta_{(i+1)}= B^*_{i+1}E^*_i.
\end{equation}
Notice  \reff{eq:BE*} holds for $i=d$ thanks to the conventions. We get on $\Psi_i^\delta$:
\[
\delta_{(i)}' -\delta_{(i+1)}'
= -K'_{i+1} B^*_{i+1}E^*_i+B^*_{i+1}a^*_iE^*_{i-1}
= -\delta'_{(i+1)} + B^*_{i+1}a^*_iE^*_{i-1}.
\]
where we took  the
derivative in \reff{eq:BE*}, twice the induction hypothesis for
$B^*_{i+1}$ and \reff{eq:Bi_exp} for the first equality; then
\reff{eq:F} and \reff{eq:BE*} for the second.
We deduce that on $\Psi_i^\delta$:
\begin{equation}
   \label{eq:BE*2}
\delta_{(i)}'= B^*_{i+1}a^*_iE^*_{i-1}.
\end{equation}
 On $\Psi_i^\delta$, we can divide \reff{eq:BE*2} by \reff{eq:BE*}
 and get, thanks to \reff{eq:F}:
\[
    \frac{a^*_i B^*_{i+1}}{B^*_i}=\frac{\delta'_{(i)}}{\delta_{(i-1)}
      -\delta_{(i)}}= K'_i.
\]
Notice that $(B_i^*)'= - a^*_i B^*_{i+1}$. So using  the representation
\reff{eq:Bi_exp} of $B_i$,
we get that \reff{eq:HRB} holds
for  $i$. Thus \reff{eq:HRB} holds for $1\leq i\leq d+1$.
Then use \reff{eq:HRB} as well as  $(B_i^*)'= -a^*_i B^*_{i+1}$ and  $B_i'= -a_i B_{i+1}$
to get that for  $t \in (g_i^{(j)}, d_i^{(j)}) \cap (g_{i+1}^{(k)}, d_{i+1}^{(k)})$:
\[
a^*_i(t) = \frac{B^*_i(m_i^{(j)})}{B^*_{i+1}(m_{i+1}^{(k)})} a_i(t).
\]
Therefore if $u=(u_1, \hdots, u_d) \in L^\delta$, we have:
\[
   \prod_{i=1}^d a^*_i(u_{(i)}) = \frac{B^*_1(m_1)}{B^*_{d+1}(m_{d+1})} \prod_{i=1}^d a_i(u_{(i)}),
\]
since when $u \in L_\delta$, $u_{(i-1)}$ and $u_{(i)}$ belong to the same interval $(g_i^{(j)}, d_i^{(j)})$ for $2 \leq i \leq d$.
This ensures that $c_\delta$ and $c^*$ are
densities of probability function which differ by a multiplicative
constant, therefore they are equal. This ends the proof of Proposition
\ref{prop:c_delta_unique_product}.

\subsection{Proof of case (a) for Theorems \ref{theo:spec}  and
  \ref{theo:f*}}
\label{sec:case-a}

We first consider the case $d=2$. Let $\delta\in \cd^0$ with
$\cj(\delta)=+\infty $. Recall $\cj(\delta)$ is defined by
\reff{eq:cj_delta}. We have:
\begin{align*}
\cj(\delta)
&=- \int_I \delta'_{(2)}(t) \log(2(t-\delta_{(2)}(t)))\, dt\\
&=-\log(2) - \int_I  \log(t-\delta_{(2)}(t))\, dt+ \int_I
(1-\delta'_{(2)}(t)) \log(t-\delta_{(2)}(t))\, dt \\
&=-\log(2) - \int_I  \log(t-\delta_{(2)}(t))\, dt+
\left[
(t-\delta_{(2)}(t)) \log(t-\delta_{(2)}(t)) -
(t-\delta_{(2)}(t))
\right]_0^1\\
&=-\log(2) - \int_I  \log(t-\delta_{(2)}(t))\, dt,
\end{align*}
where we  used $\delta_{(1)}+  \delta_{(2)}=2t$ for the  first equality,
$\delta_{(2)}(1)=1$ and $\delta_{(2)}(0)=0$ for the second and last.  In
particular, we  obtain that  $\cj(\delta)$ is  equal to  $-\log(2)+ \cjj
(\delta_{(2)})$,     with      $\cjj$  as also defined     by (1)     in
\cite{butucea2013maximum}.  Therefore  we  deduce case  (a)  of  Theorem
\ref{theo:spec}  (for   $d=2$)  from   case  (a)   of  Theorem   2.4  in
\cite{butucea2013maximum}.  Then, we  get from \reff{eq:entr_decomp} and
Theorem \ref{theo:spec}  case (a) that $H(F)  = -\infty$ for all  $F \in
\cl_{2}^{OS}(\mathbf{F})$.   This    proves   case   (a)    for  Theorem
\ref{theo:f*} (for $d=2$).

  We  then consider  the case  $d\geq  2$.  Let  $\delta\in \cd^0$  with
  $\cj(\delta)=+\infty $.  This implies that there exists  $2 \leq i \leq d$
  such  that  $\int_I \delta_{(i)}'(t)  \val{\log(\delta_{(i-1)}(t)  -
  \delta_{(i)}(t))}  \,  dt= + \infty $.  Set  $\mathbf{F}=(\delta_{(i-1)},
  \delta_{(i)})$  and notice  that  $\mathbf{F}$ belongs  to $\cf_2$  as
  $\delta_{(i)}$  is   $d$-Lipschitz.  Since   $\int_I  \delta_{(i)}'(t)
  \val{\log(\delta_{(i-1)}(t) - \delta_{(i)}(t))} \,  dt=+\infty $, we deduce
  from   the  first   part  of   this   Section  that   $\max  _{F   \in
    \cl_{2}^{OS}(\mathbf{F})} H(F)=-\infty$.

Consider a copula $C$ belonging to $\cc_\delta\bigcap \cc^{sym}$ and $U$
a random vector on $I^d$ with cdf $C$.  According to Lemma
\ref{lem:HX_HXPI} and Lemma \ref{lem:exch_OS}, as $C$ is symmetric, we have:
\[
H(U^{OS})=H(U) - \log(d!)=H(C) - \log(d!).
\]
It is easy to check that if $X=(X_1, \ldots, X_d)$ is a random vector on
$I^d$ and $2\leq i\leq d$, then we have $  H((X_{i-1}, X_i))\geq  H(X)$.
This implies that, for $V=(U^{OS}_{i-1}, U^{OS}_i)$,
\[
H(V) \geq H(C) - \log(d!).
\]
Since the  cdf of $U_\ell^{OS}$ is $\delta_{(\ell)}$ as $C\in
\cc_\delta$, we deduce the cdf
of $V$ belongs   to $\cl_{2}^{OS}(\mathbf{F}) $, and thus $H(V)=-\infty
$. This implies that $H(C)=-\infty $. Thanks to Proposition
\ref{prop:sym} which states that the entropy is maximal on symmetric
copulas, we deduce that:
\[
\max_{C \in \cc_\delta} H(C)=\max_{C \in \cc_\delta\bigcap \cc^{sym}}
H(C)= -\infty .
\]
This proves  cases (a) for  Theorem \ref{theo:spec}.  Then, we  get from
\reff{eq:entr_decomp} that $H(F) =
-\infty$ for all $F \in  \cl_{d}^{OS}(\mathbf{F})$. This proves case (a)
for Theorem \ref{theo:f*}.

\subsection{Proof of Theorem \ref{theo:spec}, case (b)}
\label{sec:proof-spec}
Let $\delta \in  \cd$ with $\cj(\delta) <  +\infty$.  Thanks to Lemma
\ref{lem:delta'_Psi},
$\cj(\delta)<+\infty $ implies that $\delta\in \cd^0$.   By construction,
$c_\delta$   introduced  in   Proposition  \ref{prop:c_delta}   verifies
$\mu-a.e.$  $c_\delta  >  0$.   The density  $c_\delta$  is  a  feasible
solution  to  the problem  $(P_\mu^\delta)$.
Theorem  \ref{theo:bor_lew_nuss} ensures
the existence of a unique  optimal solution $c^*$. Furthermore, by Lemma
\ref{lem:c*=ab},  we  have  that there  exist  non-negative,  measurable
functions  $a_i^*$,  $1  \leq  i   \leq  d$,  such  that  $c^*(u)=(1/d!)
\ind_{L_\delta}(u)\prod_{i=1}^d  a_i^*(u_{i})$  $\mu$-a.e. By  Corollary
\ref{cor:c-triangle}, the optimal solution  $c$ of $(P^\delta)$ is given
by, for $u=(u_1, \hdots, u_d)$:
   \[
     c(u)=c^*(u^{OS})\ind_{Z_\delta^c \cap L_\delta}(u)= \inv{d!} \ind_{L_\delta}(u)\prod_{i=1}^d a^*_i(u_{(i)}) \ind_{\{\delta'_{(i)}(u_{(i)}) \neq 0 \}}.
   \]
   Since      $c$     is      of      product     form,      Proposition
   \ref{prop:c_delta_unique_product}  yields   that  $c=c_\delta$  a.e.,
   therefore  $C_\delta$ is  the unique  copula achieving  $H(C_\delta)=
   \max_{C \in \cc_\delta} H(C)$.

\newpage

\section{Overview of the notations} \label{sec:not-appendix}

\begin{itemize}
 \item[-] $\cf_d$: set of continuous one-dimensional marginals cdf $\mathbf{F}=(\mathbf{F}_1, \hdots, \mathbf{F}_d)$ of $d$-dimensional order statistics, see \reff{eq:def_Fd}.
 \item[-] $\cf_d^0$: set of continuous one-dimensional marginals cdf $\mathbf{F}=(\mathbf{F}_1, \hdots, \mathbf{F}_d)$ of $d$-dimensional abs. cont. order statistics, see Definition \ref{defi:f_d^0}.

 \item[-] $\cl_d$: set of all cdf's on $\R^d$.
 \item[-] $\cl_d^{1c}$: set of cdf's on $\R^d$ with continuous one-dimensional marginals cdf.
 \item[-] $\cl_d^0$: set of absolutely continuous cdf's on $\R^d$.
 \item[-] $\cl_d^{OS}$: set of cdf's of $d$-dimensional order statistics with continuous one dimensional marginals cdf.
 \item[-] $\cl_d^{OS}(\mathbf{F})$: set of cdf's of $d$-dimensional order statistics with marginals cdf $\mathbf{F}$, see \reff{eq:def_LOSF}.
 \item[-] $\cl_d^{sym}$: set of symmetric cdf's on $\R^d$.

 \item[-] $S_\mathbf{F}$: symmetrizing operator on copulas, associated to the marginals cdf $\mathbf{F}$, see Definition \ref{defi:SF}.
 \item[-] $\cc$: set of all copulas.
 \item[-] $\cc^0$: set of absolutely continuous copulas.
 \item[-] $\cc^{OS}(\mathbf{F})$: set of copulas of order statistics with marginals cdf $\mathbf{F}$, see \reff{eq:def_COSF}.
 \item[-] $\cc^{sym}$: set of symmetric (permutation invariant) copulas.
 \item[-] $\cc^{sym}(\mathbf{F})$: image of the set $\cc^{OS}(\mathbf{F})$ by the operator $S_\mathbf{F}$, see \reff{eq:def-CFsym}. It is the set of symmetric copulas with multidiagonal
 $\delta^\mathbf{F}$.
 \item[-] $\cc_\delta$: set of copulas with multidiagonal $\delta$, see Section \ref{sec:multidiag}.
 \item[-] $\cc_\delta^0$: set of abs. cont. copulas with multidiagonal $\delta$, see Section \ref{sec:multidiag}.

 \item[-] $\cd$: set of multidiagonals of copulas, see Section \ref{sec:multidiag}.
 \item[-] $\cd^0$: set of multidiagonals of abs. cont. copulas, see Section \ref{sec:multidiag}.

 \item[-] $\Psi_i^\mathbf{F}$: set of points $t \in \R$ for which the marginals cdf $\mathbf{F}=(\mathbf{F}_1, \hdots, \mathbf{F}_d)$ verify $\mathbf{F}_{i-1}(t) > \mathbf{F_i}(t)$ , see \reff{eq:def_Psi_i_F}.
 \item[-] $T^\mathbf{F}$: set of points $u=(u_1, \hdots, u_d) \in I^d$ for which $\mathbf{F}_1^{-1}(u_1) \leq \hdots \leq \mathbf{F}_d^{-1}(u_d)$, see \reff{eq:c_zero-T}. The density of all copulas in $\cc^{OS}(\mathbf{F})$ vanishes on $T^\mathbf{F}$.
 \item[-] $L^\mathbf{F}$: set of ordered vectors $x \in \R^d$ such that the marginals cdf's $\mathbf{F}=(\mathbf{F}_1, \hdots, \mathbf{F}_d)$ verify $\mathbf{F}_{i-1}(t) > \mathbf{F}_{i}(t)$ for all $t \in (x_{i-1},x_i)$, $2 \leq i \leq d$, see \reff{eq:def_LF}. The density of any abs. cont. cdf in $\cl_d^{OS}(\mathbf{F})$ vanishes outside $L^\mathbf{F}$.
 \item[-] $L_\delta$: set of points $u=(u_1, \hdots, u_d) \in I^d$ for which all points $t \in (u_{(i-1)},u_{(i)})$ verify $\delta_{(i-1)}(t) > \delta_{(i)}(t)$ for all $2 \leq i \leq d$,
  see \reff{eq:def_LD}. The density of  any copula in $\cc^0_\delta$ vanishes outside $L_\delta$.
 \item[-] $Z_\delta$: set of points $u=(u_1, \hdots, u_d) \in I^d$  such that $\delta'_{(i)}(u_{(i)})=0$ for some $1\leq i \leq d$, see \reff{eq:def-Z}. The density of
 any copula in $\cc^0_\delta$ vanishes on $Z_\delta$.

\end{itemize}

\newpage

\bibliographystyle{abbrvnat}
\bibliography{Max_entr_ord_biblio}

\begin{thebibliography}{19}
\providecommand{\natexlab}[1]{#1}
\providecommand{\url}[1]{\texttt{#1}}
\expandafter\ifx\csname urlstyle\endcsname\relax
  \providecommand{\doi}[1]{doi: #1}\else
  \providecommand{\doi}{doi: \begingroup \urlstyle{rm}\Url}\fi

\bibitem[Arnold et~al.(1992)Arnold, Balakrishnan, and
  Nagaraja]{arnold1992first}
B.~C. Arnold, N.~Balakrishnan, and H.~N. Nagaraja.
\newblock \emph{A first course in order statistics}, volume~54.
\newblock Siam, 1992.

\bibitem[Av{\'e}rous et~al.(2005)Av{\'e}rous, Genest, and
  Kochar]{averous2005dependence}
J.~Av{\'e}rous, C.~Genest, and S.~C. Kochar.
\newblock On the dependence structure of order statistics.
\newblock \emph{Journal of Multivariate Analysis}, 94\penalty0 (1):\penalty0
  159--171, 2005.

\bibitem[Bickel(1967)]{bickel1967some}
P.~J. Bickel.
\newblock Some contributions to the theory of order statistics.
\newblock In \emph{Proc. {F}ifth {B}erkeley {S}ympos. {M}ath. {S}tatist. and
  {P}robability ({B}erkeley, {C}alf., 1965/66), {V}ol. {I}: {S}tatistics},
  pages 575--591. Univ. California Press, Berkeley, Calif., 1967.

\bibitem[Boland et~al.(1996)Boland, Hollander, Joag-Dev, and
  Kochar]{boland1996bivariate}
P.~J. Boland, M.~Hollander, K.~Joag-Dev, and S.~Kochar.
\newblock Bivariate dependence properties of order statistics.
\newblock \emph{Journal of Multivariate Analysis}, 56\penalty0 (1):\penalty0
  75--89, 1996.

\bibitem[Borwein et~al.(1994)Borwein, Lewis, and Nussbaum]{borwein1994entropy}
J.~Borwein, A.~Lewis, and R.~Nussbaum.
\newblock Entropy minimization, \( {DAD} \) problems, and doubly stochastic
  kernels.
\newblock \emph{Journal of Functional Analysis}, 123\penalty0 (2):\penalty0 264
  -- 307, 1994.
\newblock ISSN 0022-1236.
\newblock \doi{http://dx.doi.org/10.1006/jfan.1994.1089}.
\newblock URL
  \url{http://www.sciencedirect.com/science/article/pii/S0022123684710895}.

\bibitem[Butucea et~al.(2015)Butucea, Delmas, Dutfoy, and
  Fischer]{butucea2013maximum}
C.~Butucea, J.-F. Delmas, A.~Dutfoy, and R.~Fischer.
\newblock Maximum entropy copula with given diagonal section.
\newblock \emph{Journal of Multivariate Analysis}, 137:\penalty0 61 -- 81,
  2015.
\newblock ISSN 0047-259X.
\newblock \doi{http://dx.doi.org/10.1016/j.jmva.2015.01.003}.
\newblock URL
  \url{http://www.sciencedirect.com/science/article/pii/S0047259X15000081}.

\bibitem[David and Nagaraja(1970)]{david1970order}
H.~A. David and H.~N. Nagaraja.
\newblock \emph{Order statistics}.
\newblock Wiley Online Library, 1970.

\bibitem[de~Melo~Mendes and Sanfins(2007)]{de2007limiting}
B.~V. de~Melo~Mendes and M.~A. Sanfins.
\newblock The limiting copula of the two largest order statistics of
  independent and identically distributed samples.
\newblock \emph{Brazilian Journal of Probability and Statistics}, 21:\penalty0
  85--101, 2007.

\bibitem[Dubhashi and H{\"a}ggstr{\"o}m(2008)]{dubhashi2008note}
D.~Dubhashi and O.~H{\"a}ggstr{\"o}m.
\newblock A note on conditioning and stochastic domination for order
  statistics.
\newblock \emph{J. Appl. Probab.}, 45\penalty0 (2):\penalty0 575--579, 2008.
\newblock ISSN 0021-9002.
\newblock \doi{10.1239/jap/1214950369}.
\newblock URL \url{http://dx.doi.org/10.1239/jap/1214950369}.

\bibitem[Hu and Chen(2008)]{hu2008dependence}
T.~Hu and H.~Chen.
\newblock Dependence properties of order statistics.
\newblock \emph{Journal of Statistical Planning and Inference}, 138\penalty0
  (7):\penalty0 2214--2222, 2008.

\bibitem[Jaworski(2009)]{Jaworski20092863}
P.~Jaworski.
\newblock On copulas and their diagonals.
\newblock \emph{Information Sciences}, 179\penalty0 (17):\penalty0 2863 --
  2871, 2009.
\newblock ISSN 0020-0255.
\newblock \doi{10.1016/j.ins.2008.09.006}.
\newblock URL
  \url{http://www.sciencedirect.com/science/article/pii/S0020025508003836}.

\bibitem[Jaworski and Rychlik(2008)]{jaworski2008distributions}
P.~Jaworski and T.~Rychlik.
\newblock On distributions of order statistics for absolutely continuous
  copulas with applications to reliability.
\newblock \emph{Kybernetika}, 44\penalty0 (6):\penalty0 757--776, 2008.

\bibitem[Kim and David(1990)]{kim1990dependence}
S.~Kim and H.~David.
\newblock On the dependence structure of order statistics and concomitants of
  order statistics.
\newblock \emph{Journal of statistical planning and inference}, 24\penalty0
  (3):\penalty0 363--368, 1990.

\bibitem[Lebrun and Dutfoy(2014)]{lebrun2014copulas}
R.~Lebrun and A.~Dutfoy.
\newblock Copulas for order statistics with prescribed margins.
\newblock \emph{Journal of Multivariate Analysis}, 128:\penalty0 120--133,
  2014.

\bibitem[Navarro and Balakrishnan(2010)]{navarro2010study}
J.~Navarro and N.~Balakrishnan.
\newblock Study of some measures of dependence between order statistics and
  systems.
\newblock \emph{Journal of Multivariate Analysis}, 101\penalty0 (1):\penalty0
  52--67, 2010.

\bibitem[Navarro and Spizzichino(2010)]{NavarroSpizzichino2010}
J.~Navarro and F.~Spizzichino.
\newblock On the relationships between copulas of order statistics and marginal
  distributions.
\newblock \emph{Statist. Probab. Lett.}, 80\penalty0 (5-6):\penalty0 473--479,
  2010.
\newblock ISSN 0167-7152.
\newblock \doi{10.1016/j.spl.2009.11.025}.
\newblock URL \url{http://dx.doi.org/10.1016/j.spl.2009.11.025}.

\bibitem[R{\"u}schendorf and Thomsen(1993)]{Ruschendorf1993369}
L.~R{\"u}schendorf and W.~Thomsen.
\newblock Note on the {S}chr\"odinger equation and {$I$}-projections.
\newblock \emph{Statist. Probab. Lett.}, 17\penalty0 (5):\penalty0 369--375,
  1993.
\newblock ISSN 0167-7152.
\newblock \doi{10.1016/0167-7152(93)90257-J}.
\newblock URL \url{http://dx.doi.org/10.1016/0167-7152(93)90257-J}.

\bibitem[Schmitz(2004)]{schmitz2004revealing}
V.~Schmitz.
\newblock Revealing the dependence structure between ${X}_{(1)}$ and
  ${X}_{(n)}$.
\newblock \emph{Journal of statistical planning and inference}, 123\penalty0
  (1):\penalty0 41--47, 2004.

\bibitem[Zhao and Lin(2011)]{Zhao2011628}
N.~Zhao and W.~T. Lin.
\newblock A copula entropy approach to correlation measurement at the country
  level.
\newblock \emph{Applied Mathematics and Computation}, 218\penalty0
  (2):\penalty0 628 -- 642, 2011.
\newblock ISSN 0096-3003.
\newblock \doi{10.1016/j.amc.2011.05.115}.
\newblock URL
  \url{http://www.sciencedirect.com/science/article/pii/S0096300311007983}.

\end{thebibliography}

\end{document}